\documentclass[reqno,12pt]{amsart}
\usepackage{amsfonts}
\usepackage{bbm}
\usepackage{}
\numberwithin{equation}{section}
\usepackage{hyperref}
\usepackage{color}
\usepackage{caption}
\usepackage{rotating}
\usepackage{indentfirst}
\usepackage{color}
\usepackage{amssymb}
\usepackage{mathrsfs}
\usepackage[color,matrix,arrow,pdf]{xy}
\usepackage{rotating}
\usepackage{tikz, physics}

\xyoption{all}
\def\Ext{\mbox{\rm Ext}} \def\Tor{\mbox{\rm Tor}}\def\Hom{\mbox{\rm Hom}}  
    \def\Mod{\mbox{\rm Mod}}
\def\Ker{\mbox{\rm Ker}}  \def\Coker{\mbox{\rm Coker}}

 \def\Id{\mbox{\rm Id}}
\def\Im{\mbox{\rm Im}} \def\id{\mbox{\rm id}}  
 \def\pd{\mbox{\rm pd}} 
\def\sup{\mbox{\rm sup}}

\def\Ch{\mbox{\rm Ch}}

\renewcommand{\baselinestretch}{1.2}

\theoremstyle{plain}
\newtheorem{theorem}{\bf Theorem}[section]
\newtheorem{lemma}[theorem]{\bf Lemma}
\newtheorem{corollary}[theorem]{\bf Corollary}
\newtheorem{proposition}[theorem]{\bf Proposition}

\theoremstyle{definition}
\newtheorem{definition}[theorem]{\bf Definition}
\newtheorem{remark}[theorem]{\bf Remark}
\newtheorem{example}[theorem]{\bf Example}

\newcommand{\bt}{\begin{theorem}}
\newcommand{\et}{\end{theorem}}
\newcommand{\bl}{\begin{lemma}}
\newcommand{\el}{\end{lemma}}
\newcommand{\bd}{\begin{definition}}
\newcommand{\ed}{\end{definition}}
\newcommand{\bc}{\begin{corollary}}
\newcommand{\ec}{\end{corollary}}
\newcommand{\bp}{\begin{proof}}
\newcommand{\ep}{\end{proof}}
\newcommand{\bx}{\begin{example}}
\newcommand{\ex}{\end{example}}
\newcommand{\br}{\begin{remark}}
\newcommand{\er}{\end{remark}}
\newcommand{\be}{\begin{equation}}
\newcommand{\ee}{\end{equation}}
\newcommand{\ba}{\begin{align}}
\newcommand{\ea}{\end{align}}
\newcommand{\bn}{\begin{enumerate}}
\newcommand{\en}{\end{enumerate}}
\newcommand{\bcs}{\begin{cases}}
\newcommand{\ecs}{\end{cases}}
\newcommand{\CE}{\mathrm{C}\text{-}\mathrm{E}}
\newcommand{\GProj}{\mathrm{GProj}}
\newcommand{\GInj}{\mathrm{GInj}}
\newcommand{\RNum}[1]{\uppercase\expandafter{\romannumeral #1\relax}}

\makeatletter
\renewcommand{\section}{\@startsection{section}{1}{0mm}
  {-\baselineskip}{0.5\baselineskip}{\bf\leftline}}
\makeatother

\begin{document}
\begin{center}
{\large \bf Balanced pairs of Cartan-Eilenberg complexes over virtually Gorenstein rings}
\footnote{ The first author was supported by grant 22004/PI/22 funded by Fundaci\'on S\'eneca, Agencia de Ciencia y Tecnolog\'ia de la Regi\'on de Murcia. The first and second authors were supported by grant PID2024-155576NB-I00 funded by MICIU/AEI/10.13039/501100011033 /FEDER, UE. The third author was supported by China Scholarship Council (Grant No. 202506190133). }

\vspace{0.1cm}
Sergio Estrada$^{1}$, Alina Iacob$^{2}$ and Yucheng Wang$^{3,*}$ \footnote{$^*$ Corresponding author.}\\
{\renewcommand{\baselinestretch}{1}\fontsize{9}{9.5}\selectfont\itshape
$^{1}$Department of Mathematics, University of Murcia, Campus de Espinardo, 30100 Murcia, Spain\\
$^{2}$Department of Mathematical Sciences, Georgia Southern University, Statesboro, GA 30460, USA\\
$^{3}$School of Mathematics, Nanjing University, Nanjing 210093, Jiangsu Province, P.R. China\\
E-mail: $^{1}$sestrada@um.es; $^{2}$aiacob@georgiasouthern.edu;\\
$^{3}$wangyucheng2358@163.com}

\vspace{0.35cm}
{\small\itshape Dedicated to Edgar E. Enochs, with deep admiration and gratitude,\\
in recognition of his profound and lasting influence on homological algebra.}
\end{center}




\centerline {\bf  Abstract}
\bigskip
\leftskip10truemm \rightskip10truemm
Let $R$ be a ring and let $\Ch(R)$ be the category of complexes of left $R$-modules. We consider the C-E (Cartan-Eilenberg)
exact structure on $\Ch(R)$ and show that it is an efficient exact category. We prove that if $R$ is a
left virtually Gorenstein ring, the pair $(\text{C-E}({\rm GProj}),\text{C-E}({\rm GInj}))$ is a C-E-admissible
balanced pair. Under suitable additional hypotheses, the converse holds. We also consider the C-E version of Tate cohomology
and establish C-E analogues of the Avramov-Martsinkovsky exact sequences and balance results for both $\Ext$ and $\Tor$.

\leftskip0truemm \rightskip0truemm

\section{Introduction}

Cartan--Eilenberg projective and injective complexes were introduced by
Verdier in his study of derived categories \cite{V96}. They provide
resolutions of complexes which are more sensitive than ordinary degreewise
projective or injective resolutions, since they control not only the terms of a
complex but also its cycles, boundaries, cokernels and homology modules. Enochs
developed this point of view further in \cite{E11}, where the
Cartan--Eilenberg exact structure on the category of chain complexes was studied
systematically. Since then, Cartan--Eilenberg Gorenstein projective, injective
and flat complexes, as well as their associated homological dimensions, have
been studied by several authors; see, for instance,
\cite{EFHO26, YL14-1, YL14}.

Let $R$ be a ring and let $\Ch(R)$ denote the category of complexes of left
$R$-modules. The Cartan--Eilenberg exact structure on $\Ch(R)$, denoted
throughout this paper by C-E, is defined by requiring a short exact sequence of
complexes
$$
0 \longrightarrow X \longrightarrow Y \longrightarrow Z \longrightarrow 0
$$
to be exact not only degreewise, but also after applying the cycle, boundary,
cokernel and homology functors. Thus the exact category $(\Ch(R),\text{\rm C-E})$
is naturally adapted to homological questions in which the behaviour of cycles,
boundaries and homology has to be controlled simultaneously.

From the point of view of approximation theory, it is also important that the chosen exact structure be sufficiently compatible with set-theoretic constructions. In particular, one wants $(\Ch(R),\CE)$ to be efficient in the sense of \v{S}\v{t}ov\'{\i}\v{c}ek \cite{S13}, following the earlier framework of Saor\'{\i}n and \v{S}\v{t}ov\'{\i}\v{c}ek \cite{SS11}. This is proved in Proposition \ref{prop-1}. This allows to transfer the Eklof-Trlifaj set-theoretic approximation machinery, originally developed for module categories with their usual exact structure, to exact categories such as $(\Ch(R),\CE)$. This justifies why the Cartan--Eilenberg exact structure is a useful framework for constructing special approximations, complete cotorsion pairs and model structures in categories of complexes.

A further motivation for using the $\CE$-exact structure comes from the theory of
ghost ideals and from the abelian version of the generalized generating
hypothesis, as recently developed by Estrada, Fu, Herzog and Odaba\c s\i\ in \cite{EFHO26}. Recall that a chain map $f\colon X\to Y$ is called a ghost if $H_n(f)=0$
for every $n\in\mathbb Z$. In \cite[Section 8]{EFHO26}, the authors showed that the Cartan--Eilenberg exact structure is
closely related to the ghost ideal on $\Ch(R)$. In particular,
Cartan--Eilenberg projective and injective complexes provide the natural
projective and injective objects for studying this ideal from the point of view
of exact categories. The same work also shows that suitable nilpotence and convergence properties of the ghost ideal detect classical ring-theoretic information, in particular bounds on the left global dimension of $R$. Thus the Cartan--Eilenberg exact structure is not just a technical refinement of the usual exact structure on complexes; it is a setting in which ghost maps, generating hypotheses, approximation theory and global dimension interact.

The present paper continues this line of investigation from the viewpoint of
relative and Gorenstein homological algebra. Our aim is to study balanced pairs
and Tate cohomology in the exact category $(\Ch(R),\text{\rm C-E})$, replacing
ordinary projective and injective objects by Cartan--Eilenberg Gorenstein
projective and injective complexes. More precisely, following \cite[Definition 3.1]{E11}, for a class $\mathcal X$
of $R$-modules we write $$
\text{\rm C-E}(\mathcal X)
=
\{\,X\in \Ch(R)
   \mid X_n,Z_n(X),B_n(X),C_n(X),H_n(X)\in \mathcal X
      \text{ for all } n\in \mathbb Z\,\}.$$
In particular, $\CE(\mathcal P)$ and
$\CE(\mathcal I)$ are the classes of C-E projective and C-E
injective complexes, respectively, where $\mathcal P$ and $\mathcal I$
denote the classes of projective and injective $R$-modules. This notation
allows us to transfer module-theoretic cotorsion pairs to the
Cartan--Eilenberg exact category of complexes.

Balanced pairs originate from the classical fact that the derived functors
$\Ext^n_R(M,N)$ may be computed either by using a projective resolution of
$M$ or an injective resolution of $N$. Relative versions of this phenomenon
were developed in order to determine when a relative derived functor is
independent of the chosen side of the resolution. In an abelian category, Chen \cite{C10}
proved that if $(\mathcal X,\mathcal Y)$ is a balanced pair, then the
corresponding relative derived functor $
\Ext^n_{(\mathcal X,\mathcal Y)}(M,N)$
can be computed either from a left $\mathcal X$-resolution of $M$ or from a
right $\mathcal Y$-resolution of $N$. The same philosophy applies
in exact categories with enough projective and injective objects.

Our first main result shows that, over a virtually Gorenstein ring, the
Cartan--Eilenberg Gorenstein projective and injective complexes form such a
balanced pair.

 \begin{theorem}
        (= Theorem \ref{thm-1}) Over any left virtually Gorenstein ring $R$, the pair
    $$(\CE({\rm GProj}),\CE({\rm GInj}))$$ is a $\CE$-admissible balanced pair.
    \end{theorem}

 This extends \cite[Theorem 2.22]{YL14}, where the corresponding result was
proved under the stronger assumption that the ring has finite global Gorenstein
projective dimension. The virtual Gorenstein condition is the natural hypothesis
here, because it identifies the two perpendicular classes  ${\rm GProj}^{\perp}={}^{\perp}{\rm GInj},$
and this allows the module Gorenstein cotorsion pairs to be lifted
to the category of complexes with the $\CE$-exact structure.

We also obtain a converse under a completeness assumption.

   \begin{theorem}
        (= Theorem \ref{thm-2}) If the cotorsion pair $({\rm GProj},{\rm GProj}^{\bot})$ is complete, then the following statements are equivalent.

    {\rm (1)} $R$ is a left virtually Gorenstein ring.

    {\rm (2)} The pair $(\CE({\rm GProj}),\CE({\rm GInj}))$ is a $\CE$-admissible balanced pair in the exact category $(\Ch(R),\CE)$.

    {\rm (3)} The pair $({\rm GProj},{\rm GInj})$ is an admissible balanced pair in $\Mod R$.
    \end{theorem}

The second part of the paper is devoted to Tate cohomology and Tate homology in
the Cartan--Eilenberg setting. Tate cohomology was introduced by Tate \cite{T52} in the
context of class field theory and was later generalized by Avramov and
Martsinkovsky to Gorenstein homological algebra \cite{AM02}.
Iacob proved a balance theorem for generalized Tate cohomology \cite{I05}.
Here we develop the corresponding C-E version. If a complex $M$ admits a complete C-E
projective resolution $ T \longrightarrow P \longrightarrow M,$
we define the C-E Tate cohomology groups by $\widehat{\Ext}^{\,n}_{\text{\rm C-E}}(M,N)=H^n(\operatorname{Hom}(T,N)).$
Dually, complete C-E injective resolutions give the corresponding injective
version. We prove that, whenever
\[
T\longrightarrow P\longrightarrow M
\qquad\text{and}\qquad
N\longrightarrow I\longrightarrow U
\]
are complete C-E projective and complete C-E injective resolutions, respectively, the cohomology groups of $\Hom(T,N)$ and $\Hom(M,U)$ are naturally isomorphic. Therefore,
$$
\widehat{\Ext}^{\,n}_{\CE}(M,N)
\cong
\overline{\Ext}^{\,n}_{\CE}(M,N)
\qquad(n\in\mathbb Z).$$

We also compare our construction with Gillespie's model-categorical approach to Tate derived functors (see \cite{G21}). In the virtually Gorenstein case, the relevant Cartan--Eilenberg cotorsion pairs give exact model structures, and the equality \(\GProj^{\perp}={}^\perp\GInj\) identifies the two middle classes. This comparison shows that our C-E Tate cohomology agrees with the corresponding relative derived functors obtained from exact model structures by applying Gillespie's techniques, while our definition is given directly in terms of complete Cartan--Eilenberg resolutions.

Finally, we establish C-E analogues of the Avramov--Martsinkovsky exact sequences. For instance, if $M$ has finite C-E Gorenstein projective dimension, then there is a long exact sequence relating relative C-E cohomology, ordinary C-E cohomology and C-E Tate cohomology. We also discuss the corresponding C-E Tate homology theory for $\Tor$ and prove a balance result under suitable left and right virtual Gorenstein hypotheses.

The paper is organized as follows. Section~2 recalls the necessary background on exact categories, Cartan--Eilenberg exact sequences, C-E projective and injective complexes, Gorenstein homological algebra and virtually Gorenstein rings. Section~3 studies balanced pairs in exact categories and then specializes to $(\Ch(R),\CE)$. Section~4 introduces $\CE$ Tate cohomology and C-E Tate homology, proves the corresponding balance theorems (see Theorems \ref{thm-4} and \ref{thm-5}), and establishes Cartan--Eilenberg versions of the Avramov--Martsinkovsky exact sequences (see Corollary \ref{coro-3} and Theorem \ref{thm-3}).

\section{Preliminaries}
Throughout this paper, all rings are associative rings with $1$. Let $R$ be a ring. We use $\Mod R$ to
denote the category of left $R$-modules,
and all subcategories of $\Mod R$ involved are full and closed under isomorphisms.
\subsection{Exact categories}
Throughout this paper, we will work in the context of an exact category. Let $\mathcal{C}$ be an additive category. Recall from \cite{B10} that an exact category
$(\mathcal{C},\mathcal{E})$ is an isomorphism-closed collection of
distinguished kernel-cokernel pairs $(i,p)$ in $\mathcal{C}$ $$\xi :A\stackrel{i}\to B
\stackrel{p}\to C,$$ where $i$ and $p$ are called, respectively, the {\it inflation} and the {\it deflation} of
$\xi$, respectively, satisfying:

{\rm (i)} For every object $C\in\mathcal{C}$, the identity morphism $1_{C}$ is both an inflation and a deflation;

{\rm (ii)} Inflations and deflations are closed under composition;

{\rm (iii)} The class $\mathcal{E}$ is closed under pushout and pullback along any morphism in $\mathcal{C}$.

A kernel-cokernel pair $(i,p)$, or $\xi$, in $\mathcal{E}$ is called a {\it conflation} in
$(\mathcal{C},\mathcal{E})$. We will also sometimes say that $C$ is the cokernel of $i$ and
use the notation $C=B/A$; $A$ is the kernel of $p$, which may be denoted by $A\subseteq B$.
For given objects $A$ and $C$ in $\mathcal{C}$, the axioms of exact categories allow
for the usual construction of the Yoneda Ext bifunctor $\Ext^1_{\mathcal{C}}(M,N)$. It is the
abelian group of equivalence classes of conflations $N\rightarrowtail Z\twoheadrightarrow M$
with respect to the {\it Baer sum} operation. In particular, we obtain that
$\Ext^1_{\mathcal{C}}(M,N)=0$ iff every conflation $N\rightarrowtail Z\twoheadrightarrow M$ is
isomorphic to $N\rightarrowtail N\oplus M\twoheadrightarrow M$. In general, we say a conflation
$\xi:A\stackrel{i}\to B \stackrel{p}\to C$ in $\mathcal{E}$ is {\it trivial} if it represents
$0$, that is, the morphism $i$ is a section: there exists a morphism $r:B\to A$ such that
$ri=1_A$; or equivalently, $p$ is a retraction: there exists a morphism $s:C\to B$ such
that $ps=1_C$; see \cite[Remark 7.4]{B10}.
We say an object $I\in\mathcal{C}$ is {\it injective} if any inflation $I\rightarrowtail Z$
splits. It follows that $I$ is injective if and only if $\Ext^1_{\mathcal{C}}(N,I)=0$ for
any $N\in\mathcal{C}$, or equivalently, $\Hom_{\mathcal{C}}(-,I)$ is exact.
{\it Projective objects} in $\mathcal{C}$ are defined dually. An exact category $\mathcal{C}$
is said to have {\it enough projectives} if for every object $M\in\mathcal{C}$ there exists a
projective object $P$ and a deflation $P\to M$, see \cite[Definition 11.9]{B10}. The notion
of having {\it enough injectives} is defined dually.
An exact category is said to be {\it weakly idempotent complete} (WIC for short) if every
retraction has a kernel, see \cite[Section 7]{B10}.

\subsection{Cartan-Eilenberg complexes and Cartan-Eilenberg short exact sequences}
Let $\Ch (R)$ be the category whose objects are complexes of left
$R$-modules of the following form $X^{\bullet}=((X_n),(d_n))$: $$ \cdots\to X_{n+1}
\stackrel{d_{n+1}}\longrightarrow X_n\stackrel{d_{n}}\longrightarrow X_{n-1}\to \cdots$$ with $d_{n}d_{n+1}=0$ for each
$n\in\mathbb{Z}$ and whose morphisms are the chain maps. If $X^{\bullet}$ is a complex, we
denote by $\Sigma(X^{\bullet})$ the suspension of $X^{\bullet}$ with
$\Sigma(X^{\bullet})_n=X_{n-1}$ and $d^{\Sigma(X^{\bullet})}_n=-d_{n-1}^{X^{\bullet}}$.
For a complex $X^{\bullet}$, we write $\tau_{<n}X^{\bullet}$ for the {\it truncations} of
$X^{\bullet}$, i.e.,
\begin{equation*}
    \tau_{<n}X^{\bullet}_i=\begin{cases}
        X^{\bullet}_i, & i<n, \\
        0, & i\geq n.
    \end{cases} \quad \text{and} \quad     d_i^{\tau_{<n}X^{\bullet}}=\begin{cases}
        d^{X^{\bullet}}_i, & i<n, \\
        0, & i\geq n.
    \end{cases}
\end{equation*}
The notion of $\tau_{>n}X^{\bullet}$ is defined dually.
We denote by $\Sigma^k(X^{\bullet})$ the $k$-th suspension of $X^{\bullet}$ where
$k\in\mathbb{Z}$. If $M$ is a left $R$-module, then $M$ can be thought of as a complex
concentrated at degree $0$ called the {\it stalk complex of $M$}.
Let $S^0(M)$ be the stalk complex of $M$, i.e. $S^0(M)=\cdots \to 0\to M\to 0\to \cdots$.
Similarly, let $D^0(M)$ be the complex $\cdots\to 0\to M\stackrel{1}\to M\to 0\to \cdots$ with
the two $M$'s in the $1$-st and $0$-th place and let $S^n(M):=\Sigma^n(S^0(M))$ and
$D^n(M):=\Sigma^n(D^0(M))$ for each $n\in\mathbb{Z}$.
For a complex $X^{\bullet}$, we write $Z_n(X^{\bullet})=\Ker(X_n\stackrel{d_{n}}\to
X_{n-1})$ and $B_n(X^{\bullet})=\Im(X_n\stackrel{d_{n}}\to X_{n-1})$ for the $n$-th cycle
and the $n$-th boundary, respectively. Similarly, we also denote by
$C_n(X^{\bullet})=\Coker(X_n\stackrel{d_{n}}\to X_{n-1})$ and
$H_n(X^{\bullet})=Z_n(X^{\bullet})/B_{n+1}(X^{\bullet})$ the $n$-th cokernel and the $n$-th
homology, respectively. For any subcategory $\mathcal{X}$ of $\Mod R$, the class of all
complexes with each degree in $\mathcal{X}$ is denoted by $\Ch(\mathcal{X})$.

\begin{proposition}\cite[Lemma 5.2]{E11}
    Let $\xi:0\to X^{\bullet}\to Y^{\bullet}\to Z^{\bullet}\to 0$ be a short exact sequence of chain complexes of left $R$-modules. The following are equivalent:

    {\rm (i)} For every $n\in\mathbb{Z}$, the sequence $Z_n(\xi)$ is exact.

    {\rm (ii)} For every $n\in\mathbb{Z}$, the sequence $B_n(\xi)$ is exact.

    {\rm (iii)} For every $n\in\mathbb{Z}$, the sequence $C_n(\xi)$ is exact.

    {\rm (iv)} For every $n\in\mathbb{Z}$, the sequence $H_n(\xi)$ is exact.
\end{proposition}

\begin{definition}
    We say a short exact sequence $0\to X^{\bullet}\to Y^{\bullet} \to Z^{\bullet}\to 0$ in
    $\Ch (R)$ is {\it Cartan-Eilenberg exact} ($\CE$-exact for short) if $0\to X_n\to Y_n
    \to Z_n\to 0$ and $0\to Z_n(X^{\bullet})\to Z_n(Y^{\bullet})\to Z_n(Z^{\bullet})\to 0$ are
    exact.
\end{definition}
Let $\text{C-E}$ be the class of all C-E short exact sequences in $\Ch (R)$. It is easy to
verify that $\text{C-E}$ is an exact structure on $\Ch(R)$ and we denote such exact
category by $(\Ch(R),\text{C-E})$.

\begin{proposition}\cite[Propositions 3.3 and 3.4]{E11}\label{prop-9}
    Let $X^{\bullet}\in\Ch(R)$. Then $X^{\bullet}$ is a projective object in $(\Ch(R),\CE)$,
    called {\it $\CE$ projective}, if and only if for all $n\in\mathbb{Z}$, all $X_n$,
    $Z_n(X^{\bullet})$, $B_n(X^{\bullet})$, $C_n(X^{\bullet})$ and $H_n(X^{\bullet})$ are
    projective left $R$-modules if and only if it is isomorphic to a chain complex of the form
    $$\bigoplus_{n\in\mathbb{Z}} D^n(P_n)\oplus \bigoplus_{n\in\mathbb{Z}}S^n(Q_n),$$
    where $P_n$ and $Q_n$ are projective left $R$-modules for all $n\in\mathbb{Z}$.

    Similarly, $X^{\bullet}$ is an injective object in $(\Ch(R),\CE)$, called {\it $\CE$ injective},
    if and only if for all $n\in\mathbb{Z}$, all $X_n$, $Z_n(X^{\bullet})$, $B_n(X^{\bullet})$,
    $C_n(X^{\bullet})$ and $H_n(X^{\bullet})$ are injective left $R$-modules if and only if it
    is isomorphic to a chain complex of the form $$\prod _{n\in\mathbb{Z}} D^n(I_n)\oplus
    \prod_{n\in\mathbb{Z}}S^n(E_n),$$ where $I_n$ and $E_n$ are injective left $R$-modules for
    all $n\in\mathbb{Z}$.
\end{proposition}

More generally, for any class $\mathcal{X}$ in $\Mod R$, we define
    $$ \text{C-E}(\mathcal{X})=\{X^{\bullet}\in\Ch(R)\,\mid\, X^{\bullet}_n,Z_n(X^{\bullet}),B_n(X^{\bullet}),C_n(X^{\bullet}),H_n(X^{\bullet})\in\mathcal{X}\, \text{for all} \,n\in\mathbb{Z}\}.$$
We note that if we denote by $\mathcal{P}$ and $\mathcal{I}$ the classes of projective and
injective modules, respectively, then $\text{C-E}(\mathcal{P})$ and
$\text{C-E}(\mathcal{I})$ are just the classes of C-E projective and C-E injective
complexes, respectively.

\subsection{Gorenstein projective modules and virtually Gorenstein rings}

We recall from \cite{EJ11-1} that a module $G\in\Mod R$ is said to be {\it Gorenstein projective} if there exists an exact sequence of modules
$$ \cdots\to P_{1}\to P_{0}\to P^1\to P^2\to\cdots,\eqno{(*)}$$
such that $G=\Im(P_0\to P^1)$ with each $P_i,P^i$ a projective module, which remains exact after
applying $\Hom(-,P)$ for an arbitrary projective module $P$. Similarly, we define the
notion of {\it Gorenstein injective modules}. A module is {\it Gorenstein flat} if it is a
syzygy in an acyclic complex $(*)$ consisting of flat modules which remains exact after
tensoring with an arbitrary injective right $R$-module $I$. The classes of all Gorenstein
projective, injective and flat modules are denoted by ${\rm GProj}$, ${\rm GInj}$ and
${\rm GFlat}$, respectively.
Let $\mathcal{X}$ be a subclass of $\Mod R$. The {\it $\mathcal{X}$-projective dimension}, $\mathcal{X}$-$\pd M$, of $M$ is
defined as $$\inf\{n\mid \text{there exists an exact sequence\,} 0 \to X_n\to \cdots\to X_1\to X_0 \to M\to 0$$
$$ \text{in\,} \Mod R \text{\,with all\,} X_i\in\mathcal{X}\},$$ and set $\mathcal{X}$-$\pd M=\infty$
if no such integer exists. Dually, the {\it $\mathcal{X}$-injective dimension}, $\mathcal{X}$-$\id M$, of $M$ is defined as
$$\inf\{n\mid \text{there exists an exact sequence\,} 0 \to M\to X^0\to X^1\to\cdots \to X^n\to 0$$
$$\text{in $\Mod R$ with all $X^i\in\mathcal{X}$}\},$$ and set $\mathcal{X}$-$\id M=\infty$
if no such integer exists. If we take $\mathcal{X}$ to be the class of Gorenstein
projective modules, then we denote ${\rm Gpd}M$ instead of ${\rm GProj}\text{-}\pd M$ and
let ${\rm GProj}_n$ and ${\rm GProj}_{\infty}$ be the class of the modules whose Gorenstein
projective dimensions are at most $n$ and finite, respectively.
By \v{S}aroch and \v{S}\v{t}ov\'{i}\v{c}ek \cite[Theorem 5.6]{SS20}, we know that the
cotorsion pair $({^{\perp}{\rm GInj}},{\rm GInj})$ is always complete for any ring $R$.
Also by \cite[Theorem 3.3]{CS25}, the pair $({\rm GProj},{\rm GProj}^{\perp})$ is a
hereditary cotorsion pair for any ring $R$. But it is an open problem whether it is
always complete for an arbitrary ring $R$. In addition, the description of the classes
${\rm GProj}^{\perp}$ and ${^{\perp}{\rm GInj}}$ remains unclear in general. We say a ring
$R$ is {\it left virtually Gorenstein} if ${\rm GProj}^{\perp}={^{\perp}{\rm GInj}}$.
The {\it right virtually Gorenstein rings} are defined dually. For left virtually
Gorenstein rings we refer the reader to \cite{BR07,BH08, WE25}.

\section{Balanced pairs in $(\Ch(R),\CE)$}
In this section, we will use cotorsion triples to construct balanced pairs in exact
categories. As an application, we focus on the exact category $(\Ch(R),\text{C-E})$ and get
some special balanced pairs. Now let $(\mathcal{C}, \mathcal{E})$ be an exact category and let
$\mathcal{X}$ be a class in $\mathcal{C}$, we write
$$^{\bot}\mathcal{X}:=\{M\in\mathcal{C}\mid\Ext_{\mathcal{C}}^1(M,X)=0\ \text{for\ any}\ X\in\mathcal{X}\},$$
$$\mathcal{X}^{\bot}:=\{M\in\mathcal{C}\mid\Ext_{\mathcal{C}}^1(X,M)=0\ \text{for\ any}\ X\in\mathcal{X}\}.$$
Then the definition of a cotorsion pair can be generalized to exact categories.

\begin{definition}\cite[Definition 2.1]{G11}
    A pair of classes $(\mathcal{X},\mathcal{Y})$ in an exact category $\mathcal{C}$ is
    called a {\it cotorsion pair} if the following conditions hold:

    {\rm (1)} $\Ext^1_{\mathcal{C}}(X,Y)=0$ for all $X\in\mathcal{X}$ and $Y\in\mathcal{Y}$.

    {\rm (2)} $^{\perp}\mathcal{Y}=\mathcal{X}$ and $\mathcal{X}^{\perp}=\mathcal{Y}$.
\end{definition}
We say that the cotorsion pair is {\it hereditary}\, \cite{G11} if $\mathcal{X}$
is closed under taking kernels of deflations between objects of $\mathcal{C}$
and if $\mathcal{Y}$ is closed under taking cokernels of inflations between
objects in $\mathcal{C}$. We note that for a WIC exact category with enough
projectives and injectives, $\mathcal{X}$ is closed under taking kernels of
deflations is equivalent to $\mathcal{Y}$ being closed under taking cokernels of
inflations, see \cite[Lemma 6.17]{S13}. We say that a cotorsion pair $(\mathcal{X},\mathcal{Y})$ is {\it cogenerated by a set} if there exists a set $\mathcal{S}$ such that $\mathcal{S}^{\perp}=\mathcal{Y}$.

A cotorsion pair $(\mathcal{X},\mathcal{Y})$ is called {\it right complete}\,\cite{L20} if for
each $C\in\mathcal{C}$ there exists a conflation $C\rightarrowtail Y\twoheadrightarrow X$ such
that $Y\in\mathcal{Y}$ and $X\in\mathcal{X}$. The notion of a {\it left complete} cotorsion
pair is defined dually. More generally, a class $\mathcal{X}$ is called {\it right complete}
if we replace $\mathcal{Y}$ by $^{\perp}\mathcal{X}$ in the above definition, and the notion
of a {\it left complete} class is defined dually. We say the cotorsion pair is {\it complete}
if it is both right complete and left complete. In fact, right completeness and left completeness are equivalent for any exact category with enough projectives and injectives,
see \cite[2.1]{SWZ25}. We call a triple $(\mathcal{X},\mathcal{Z},\mathcal{Y})$ a
{\it cotorsion triple} if the pairs $(\mathcal{X},\mathcal{Z})$ and
$(\mathcal{Z},\mathcal{Y})$ are cotorsion pairs. We say that the cotorsion triple is
{\it complete} (or {\it hereditary}) if the pairs $(\mathcal{X},\mathcal{Z})$ and
$(\mathcal{Z},\mathcal{Y})$ are both complete (or hereditary). For a subcategory $\mathcal{X}$
of $\mathcal{C}$, recall that a morphism $X\to C$ with $X\in\mathcal{X}$ is called a {\it right $\mathcal{X}$-approximation} (or {\it $\mathcal{X}$-precover}) if
$\Hom(X',X)\to\Hom(X',C)$ is surjective for any $X'\in\mathcal{X}$.
The notions of {\it left $\mathcal{X}$-approximations}
(or {\it $\mathcal{X}$-preenvelopes}) are defined dually. We say the subcategory
$\mathcal{X}$ is {\it contravariantly finite} ({\it covariantly finite}) if any object
$C\in\mathcal{C}$ has a right (left) $\mathcal{X}$-approximation. For a
subcategory $\mathcal{X}\subseteq\mathcal{C}$ and an object $C\in\mathcal{C}$,
an {\it $\mathcal{X}$-resolution} of $C$ is a complex $X$ (or $X_M$)
$:\cdots\to X_2\to X_1\to X_0\to C\to 0\to \cdots$ with each $X_i\in\mathcal{X}$ such that
it becomes exact after applying the functor $\Hom(X',-)$ for any $X'\in\mathcal{X}$. Moreover,
the complex is said to be {\it $\mathcal{E}$-exact} if it consists of conflations, i.e.,
$X_0\to C$, $X_1\to \Ker(X_0\to C)$ and $X_i\to \Ker(X_{i-1}\to X_{i-2})$, $\forall i\geq 2$
are deflations. We denote the ($\mathcal{E}$-exact) $\mathcal{X}$-resolution by
$X_{\bullet}\to C$ where $X_{\bullet}=\cdots X_2\to X_1\to X_0\to 0\to \cdots$ is the
{\it deleted $\mathcal{X}$-resolution} of $C$. The notion of ($\mathcal{E}$-exact)
$\mathcal{X}$-coresolution is defined dually. We say that a contravariantly finite subcategory
$\mathcal{X}$ is {\it $\mathcal{E}$-admissible} if there exists a right
$\mathcal{X}$-approximation which is a deflation. Dually, one can define the notion of an
{\it $\mathcal{E}$-coadmissible} covariantly finite subcategory.

\begin{definition}
    A pair $(\mathcal{X},\mathcal{Y})$ of additive subcategories in an exact category
    $(\mathcal{C},\mathcal{E})$ is called an {\it $\mathcal{E}$-admissible balanced pair} if
    the following conditions are satisfied:

    {\rm (BP0)} The subcategory $\mathcal{X}$ is $\mathcal{E}$-admissible contravariantly
    finite and $\mathcal{Y}$ is $\mathcal{E}$-coadmissible covariantly finite.

    {\rm (BP1)} For each object $M$, there is an $\mathcal{E}$-exact $\mathcal{X}$-resolution
    $X_{\bullet}\to M$ such that it is acyclic after applying the functors
    $\Hom_{\mathcal{C}}(-,Y)$ for all $Y\in\mathcal{Y}$.

    {\rm (BP2)} For each object $M$, there is an $\mathcal{E}$-exact
    $\mathcal{Y}$-coresolution $M\to Y_{\bullet}$ such that it is acyclic by applying the
    functors $\Hom_{\mathcal{C}}(X,-)$ for all $X\in\mathcal{X}$.
\end{definition}
The proof of Proposition \ref{prop-2} is similar to \cite[Proposition 2.6]{C10}, and it was first proved by Enochs, Jenda, Torrecillas and Xu in module categories.
\begin{proposition}\label{prop-2}
    Let $(\mathcal{C},\mathcal{E})$ be an exact category with enough projective and injective
    objects. Assume that $(\mathcal{X},\mathcal{Z},\mathcal{Y})$ is a complete hereditary cotorsion triple, then the pair $(\mathcal{X},\mathcal{Y})$ is an
    $\mathcal{E}$-admissible balanced pair.
\end{proposition}
\begin{proof}
    Let $M\in\mathcal{C}$. Since $(\mathcal{X},\mathcal{Z})$ is complete, we have a conflation
    $\xi:Z\to X\to M$ with $Z\in\mathcal{Z}$ and $X\in\mathcal{X}$.
    Since $Z\in\mathcal{X}^{\bot}$,
    it follows that $X\to M$ is a right $\mathcal{X}$-approximation by using the pullback in
    $\xi$. Then $\mathcal{X}$ is contravariantly finite. Dually we have that
    $\mathcal{Y}$ is covariantly finite. This completes the proof of {\rm (BP0)}.

    For {\rm (BP1)}, using the fact that the cotorsion pair $(\mathcal{X},\mathcal{Z})$ is
    complete, each $M$ has an $\mathcal{E}$ exact $\mathcal{X}$-resolution $X^{\bullet}$ with
    all cycles in $\mathcal{Z}$. We only need to prove that for a conflation
    $\xi:Z\to X\to M$ with $Z\in\mathcal{Z}$ and $X\in\mathcal{X}$ and an object
    $Y\in\mathcal{Y}$, the functor $\Hom_{\mathcal{C}}(-,Y)$ preserves the exactness of $\xi$. Note that, since $\mathcal{C}$ has enough injective objects, there exists a conflation $Z\to I\to Z'$ with
    $I$ injective in $\mathcal{C}$. Since $(\mathcal{X},\mathcal{Z})$ is hereditary and
    $Z,I\in\mathcal{Z}$, we have $Z'\in\mathcal{Z}$.
    Then we have the following diagram $$\xymatrix{
        Z \ar@{=}[d] \ar[r] & X \ar@{-->}[d] \ar[r] & M \ar@{-->}[d]\\
        Z \ar[r] & I \ar[r] & Z'\\
    }$$ Then $\Hom_{\mathcal{C}}(I,Y)\to \Hom_{\mathcal{C}}(Z,Y)$ is surjective since
    $\Ext_{\mathcal{C}}^1(Z',Y)=0$. It follows that,
    $\Hom_{\mathcal{C}}(X,Y)\to \Hom_{\mathcal{C}}(Z,Y)$ is surjective by
    the commuting diagram. The dual argument proves {\rm (BP2)}.
\end{proof}

We recall the notion of efficient exact categories, which slightly modifies the corresponding concept from \cite[Definition 2.6]{SS11}.
\begin{definition}\cite[Definition 3.4]{S13}
    An exact category $(\mathcal{C},\mathcal{E})$ is called {\it efficient} if

    {\rm (Ef0)} The category $\mathcal{C}$ is weakly idempotent complete.

    {\rm (Ef1)} Arbitrary transfinite compositions of inflations exist and are themselves inflations.

    {\rm (Ef2)} Every object of $\mathcal{C}$ is small relative to the class of all inflations.

    {\rm (Ef3)} The category $\mathcal{C}$ has a generator.
\end{definition}

By the above, $(\Ch(R),\text{C-E})$ is an exact category. Moreover, it is efficient.

\begin{proposition}\label{prop-1}
    The exact category $(\Ch(R),\CE)$ is an efficient exact category.
\end{proposition}
\begin{proof}
    Consider two morphisms $g^{\bullet}:B^{\bullet}\to C^{\bullet}$ and
    $f^{\bullet}:A^{\bullet}\to B^{\bullet}$. If $g^{\bullet}f^{\bullet}:
    A^{\bullet}\twoheadrightarrow C^{\bullet}$ is a deflation,
    then there exists a conflation $$0\to K^{\bullet}\to A^{\bullet}\to C^{\bullet}\to 0,$$
    which is a short exact sequence such that taking each cycle is still exact.
    Note that $A_n\twoheadrightarrow C_n$ and $Z_n(A^{\bullet})\twoheadrightarrow Z_n(C^{\bullet})$
    factor through $B_n\to C_n$ and $Z_n(B^{\bullet})\to Z_n(C^{\bullet})$, respectively,
    so they are both epimorphisms.
    Then it is clear that {\rm (Ef0)} holds by \cite[Proposition 7.6]{B10}.

    Now let $\lambda$ be an ordinal number, and
    $(X_{\alpha}^{\bullet},f_{\beta\alpha})_{\alpha<\beta<\lambda}$ be a directed system
    indexed by $\lambda$:
    $$\xymatrix{
        X^{\bullet}_0\, \ar@{>->}[r]^{f_{10}} &
        X^{\bullet}_1\, \ar@{>->}[r]^{f_{21}} &
        X^{\bullet}_2\, \ar@{>->}[r]^{f_{32}} &
        \cdots \ar[r] &
        X^{\bullet}_{\omega}\, \ar@{>->}[r]^{f_{\omega+1,\omega}} &
        X^{\bullet}_{\omega+1}\, \ar@{>->}[r] &
        \cdots
    }$$
    such that $f_{\alpha+1,\alpha}$ is an inflation for each
    $\alpha+1<\lambda$ and $X^{\bullet}_{\mu}=
    \lim\limits_{\substack{\longrightarrow \\ \alpha<\mu}}X_{\alpha}^{\bullet}$ for each
    limit ordinal $\mu<\lambda$.
    We use transfinite induction on the ordinal number $\lambda$ to show that
    $X^{\bullet}_0\rightarrowtail X^{\bullet}_{\alpha}$ is an inflation for each
    $\alpha<\lambda$.
    For $\alpha=0$, it is clear since $X^{\bullet}_0\stackrel{\rm Id}\to X^{\bullet}_0$
    is an inflation.
    For each $\alpha+1<\lambda$, if $X^{\bullet}_0\rightarrowtail X^{\bullet}_{\alpha}$
    is an inflation, we note
    that $X^{\bullet}_0\to X^{\bullet}_{\alpha+1}=X^{\bullet}_{0}\rightarrowtail
    X^{\bullet}_{\alpha}\stackrel{f_{\alpha+1,\alpha}}\rightarrowtail
    X^{\bullet}_{\alpha+1}$ is an inflation.
    Now let $\alpha<\lambda$ be a limit ordinal number and
    $X^{\bullet}_0\rightarrowtail X^{\bullet}_{\gamma}$ is an inflation for each
    $\gamma<\alpha$. We have the following commutative diagram
    $$\xymatrix{
       X^{\bullet}_0 \ar@{>->}[d]^{f_{00}} \ar@{=}[r] & X^{\bullet}_0
       \ar@{>->}[d]^{f_{10}} \ar@{=}[r] & X^{\bullet}_0 \ar@{>->}[d]^{f_{20}} \ar@{=}[r]
       & \cdots \ar@{=}[r] & X^{\bullet}_0 \ar@{>->}[d]^{f_{\gamma 0}} \ar@{=}[r] &
       X^{\bullet}_0 \ar@{>->}[d]^{f_{\gamma+1,0}} \ar@{=}[r] & \cdots \\
       X^{\bullet}_0 \ar@{>->}[r]^{f_{10}} & X^{\bullet}_1 \ar@{>->}[r]^{f_{21}}
       & X^{\bullet}_2 \ar@{>->}[r]^{f_{32}} & \cdots \ar[r] &
       X^{\bullet}_{\gamma} \ar@{>->}[r]^{f_{\gamma+1,\gamma}} &
       X^{\bullet}_{\gamma+1} \ar@{>->}[r] & \cdots
    }$$ This induces a morphism: $f_{\alpha 0}:=
    \lim\limits_{\substack{\longrightarrow \\ \gamma<\alpha}}f_{\gamma 0}:
    \lim\limits_{\substack{\longrightarrow \\ \gamma<\alpha}}X_{0}^{\bullet}\to
    \lim\limits_{\substack{\longrightarrow \\ \gamma<\alpha}}X_{\gamma}^{\bullet}$.
    Now we prove that $f_{\alpha 0}: X_{0}^{\bullet}\to X_{\alpha}^{\bullet}$ is an inflation.
    In each degree, $(f_{\alpha 0})_{n}=
    (X_0^{\bullet})_n\to (X_{\alpha}^{\bullet})_n$, $(f_{\alpha 0})_{n}$ is
    a colimit of $f_{\gamma 0}$ with $\gamma<\alpha$. It follows that $(f_{\alpha 0})_n$ is an
    injection for each $n\in\mathbb{Z}$. Then we get that $f_{\alpha 0}$ is an injection, and
    it is easy to see that
    $$0\to X_{0}^{\bullet}
    \to\lim\limits_{\substack{\longrightarrow \\ \gamma<\alpha}}X_{\gamma}^{\bullet}
    \to\lim\limits_{\substack{\longrightarrow \\ \gamma<\alpha}}
    (X_{\gamma}^{\bullet}/X_{0}^{\bullet})\to 0$$ is a short exact sequence in $\Ch (R)$.
    Therefore, we have a short exact sequence $0\to (X^{\bullet}_{0})_n \to
    (X^{\bullet}_{\alpha})_n \to (X^{\bullet}_{\alpha}/X_{0}^{\bullet})_n \to 0$ for each $n$.
    Now, since each $f_{\gamma+1,\gamma}$ is an inflation for each $\gamma<\alpha$,
    then $$Z_n(X^{\bullet}_0)\rightarrowtail Z_n(X^{\bullet}_1)\rightarrowtail
    Z_n(X^{\bullet}_2)\rightarrowtail \cdots \to Z_n(X_{\gamma}^{\bullet})\rightarrowtail
    Z_n(X^{\bullet}_{\gamma+1})\rightarrowtail\cdots$$ forms a directed system indexed by
    $\alpha$ in $\Mod R$ with each $Z_n(f_{\gamma+1,\gamma})$ injective.
    Using a similar approach, we obtain a short exact sequence in $\Mod R$:
    $$0\to Z_n(X_0^{\bullet})\to\lim\limits_{\substack{\longrightarrow \\ \gamma<\alpha}}
    Z_n(X_{\gamma}^{\bullet})\to\lim\limits_{\substack{\longrightarrow \\ \gamma<\alpha}}
    Z_n(X_{\gamma}^{\bullet})/Z_n(X_{0}^{\bullet})\to 0$$
    for each $n\in\mathbb{Z}$.
    Since $\lim\limits_{\substack{\longrightarrow \\ \gamma<\alpha}}Z_n(X_{\gamma}^{\bullet})
    \cong Z_n(\lim\limits_{\substack{\longrightarrow \\ \gamma<\alpha}}X_{\gamma}^{\bullet})=
    Z_n(X_{\alpha}^{\bullet})$ and direct limit functors are exact, then we
    get the desired short exact sequence
    $0\to Z_n(X_0^{\bullet})\to Z_n(X_{\alpha}^{\bullet})\to
    Z_n(X_{\alpha}^{\bullet}/X_0^{\bullet})\to 0$.
    Therefore, {\rm (Ef1)} holds.

    Since $\Mod R$ is a Grothendieck category, the category $\Ch (R)$, with
    the abelian exact structure, is also a Grothendieck category and hence an
    exact category. We denote it by $(\Ch (R), \mathcal{E})$.
    By \cite[Proposition 3.13]{S13}, the category $(\Ch(R),\mathcal{E})$ is
    efficient, then {\rm (Ef2)} holds in $(\Ch (R),\mathcal{E})$. So, for any
    complex $X^{\bullet}$, $X^{\bullet}$ is small relative to the class of all
    inflations in $\mathcal{E}$. Note that $(\Ch(R),\text{C-E})$ is an exact
    subcategory of $(\Ch (R), \mathcal{E})$, so $X^{\bullet}$ is also small
    relative to the class of all inflations in $(\Ch(R),\text{C-E})$ and
    ({\rm Ef2}) holds.

    For {\rm (Ef3)}, let
    $G:=\oplus_{n\in\mathbb{Z}} D^n(R)\bigoplus \oplus_{n\in\mathbb{Z}}S^n(R)$ and
    $X^{\bullet}$ be any complex in $\Ch(R)$, there are families of index sets
    $\{I_n\mid n\in\mathbb{Z}\}$ and $\{J_n\mid n\in\mathbb{Z}\}$ such that
    $R^{(I_n)}\twoheadrightarrow X^{\bullet}_n$ and
    $R^{(J_n)}\twoheadrightarrow Z_n(X^{\bullet})$ are epimorphisms for any
    $n\in\mathbb{Z}$. Set $I:=\sup\{|I_1|,|J_1|,|I_2|,|J_2|,\cdots\}$, where
    $|J|$ is the cardinality of the set $J$ and construct a morphism $G^{(I)}\stackrel{f}\to X^{\bullet}$. By \cite[Proposition 9.1.4]{EJ11-2}, $f$ is a deflation.  Then {\rm (Ef3)} holds.
\end{proof}

\begin{remark}\label{rmk-4}
    The proof of (Ef3) in Proposition \ref{prop-1} showed that $(\Ch(R),\text{C-E})$ has enough projectives. Using a similar argument, one obtains that $(\Ch(R),\text{C-E})$ has enough projectives and injectives.
\end{remark}

\begin{theorem}\label{thm-1}
    Over any left virtually Gorenstein ring $R$, the pair
    $$(\CE({\rm GProj}),\CE({\rm GInj}))$$ is a $\CE$-admissible balanced pair.
\end{theorem}
\begin{proof}
    Note that $({\rm GProj},\mathcal{W},{\rm GInj})$ is a hereditary cotorsion triple in
    $\Mod R$ where $\mathcal{W}={\rm GProj}^{\bot}={^{\bot}{\rm GInj}}$ by
    \cite[Proposition 3.16]{WE25}. Note that $({\rm GProj},\mathcal{W})$ is cogenerated by a
    set by \cite[Theorem 4.9 and Example 4.10]{SS20} and $(\mathcal{W},{\rm GInj})$ is
    cogenerated by a set by \cite[Theorem 5.6]{SS20}. It follows that
    $(\text{C-E}({\rm GProj}),\text{C-E}(\mathcal{W}))$ and $(\text{C-E}(\mathcal{W}),\text{C-E}({\rm GInj}))$ are hereditary cotorsion pairs cogenerated by sets, respectively, by
    \cite[Theorem 9.4]{E11}.
    Since $(\Ch(R),\text{C-E})$ is efficient by Proposition \ref{prop-1},
    they are complete by \cite[Proposition 5.8]{S13}. Therefore
    $(\text{C-E}({\rm GProj}),\text{C-E}(\mathcal{W}),\text{C-E}({\rm GInj}))$ is a
    complete hereditary cotorsion triple in $(\Ch (R), \text{C-E})$, and the statement
    follows from Proposition \ref{prop-2}.
\end{proof}

\begin{remark}\label{rmk-3}
    Yang and Liang \cite[Theorem 2.22]{YL14} showed that if $R$ is a ring of finite global
    Gorenstein projective dimension, then $(\text{C-E}({\rm GProj}),\text{C-E}({\rm GInj}))$
    is a C-E-admissible balanced pair. Theorem \ref{thm-1} extends this result. Indeed, note that rings
    of finite Gorenstein weak global dimension are left virtually Gorenstein by
    \cite[Theorem A]{DLW23} and Gorenstein weak global dimension is always less than or equal
    to Gorenstein global dimension over any ring (see \cite[Theorem 3.7]{WYSZ23}). It follows
    that a ring of finite Gorenstein global dimension is left virtually Gorenstein. In
    addition, there are examples of rings which are left virtually Gorenstein of infinite
    Gorenstein global dimension (see \cite[Example 2.15]{DLW23}). More precisely, suppose that
    $(R,\mathfrak{m})$ is a commutative noetherian local ring with $\mathfrak{m}^2=0$ and $R$
    is not Gorenstein. For example, $k[x,y]/(x^2,xy,y^2)$ is such a ring, where $k$ is a
    field. Then $R$ has infinite Gorenstein weak global dimension and is left virtually
    Gorenstein. Using \cite[Theorem 3.7]{WYSZ23} again, the ring $R$ has infinite Gorenstein
    global dimension.
\end{remark}

Suppose that $(\mathcal{C},\mathcal{E})$ is an exact category with enough projectives and
injectives and $(\mathcal{F},\mathcal{L})$ is an $\mathcal{E}$-admissible balanced pair. Then,
we can compute relative derived functors of $\Hom(-,-)$ using either of the two
resolutions. Using this we can get the following result, whose proof is similar to
\cite[Proposition 2.2]{C10} and \cite[Lemma 3.1]{EPZ20}.

\begin{lemma}\label{lem-1}
    Let $(\mathcal{F},\mathcal{L})$ be an $\mathcal{E}$-admissible balanced pair. A conflation in $(\mathcal{C},\mathcal{E})$ is $\Hom(F,-)$-exact for any $F\in\mathcal{F}$ if and only if it is $\Hom(-,L)$-exact for any $L\in\mathcal{L}$.
\end{lemma}

\begin{proposition}\label{prop-7}
    Suppose that $(\mathcal{C},\mathcal{E})$ is a WIC exact category with enough
    projectives and injectives. Let $\mathcal{F}$ and $\mathcal{L}$ be classes of objects in
    $\mathcal{C}$ which are closed under direct summands such that

    {\rm (1)} $\mathcal{F}$ is left complete and closed under extensions and kernels of deflations, and $\mathcal{L}$ is right complete and closed under extensions and cokernels of inflations.

    {\rm (2)} $\mathcal{F}\cap\mathcal{F}^{\bot}\subseteq {^{\bot}\mathcal{L}}$, ${^{\bot}\mathcal{L}}\cap\mathcal{L}\subseteq \mathcal{F}^{\bot}$.

    {\rm (3)} $(\mathcal{F},\mathcal{L})$ is an $\mathcal{E}$-admissible balanced pair.

Then there exists a complete hereditary cotorsion triple $(\mathcal{F},\mathcal{G},\mathcal{L})$ in $\mathcal{C}$. In this case, $\mathcal{F}\cap\mathcal{F}^{\bot}={\rm Proj}(\mathcal{C})$ and ${^{\bot}\mathcal{L}}\cap\mathcal{L}={\rm Inj}(\mathcal{C})$, where ${\rm Proj}(\mathcal{C})$ and ${\rm Inj}(\mathcal{C})$ are the classes of $\mathcal{E}$-projective and $\mathcal{E}$-injective objects, respectively.
\end{proposition}

\begin{proof}
    Note that $(\mathcal{F},\mathcal{F}^{\bot})$ is a complete hereditary cotorsion pair. Indeed, for any $X\in{^{\bot}(\mathcal{F}^{\bot})}$, there exists a conflation $V\to F\to X$ with $F\in\mathcal{F}$ and $V\in{\mathcal{F}^{\bot}}$, so the conflation is split and then $X\in\mathcal{F}$. Similarly, $({^{\bot}\mathcal{L}},\mathcal{L})$ is also a complete cotorsion pair. For any $H\in\mathcal{F}^{\bot}$, there exists a $\Hom(\mathcal{F},-)$-exact conflation $H_0\to F\to H$ with $F\in\mathcal{F}$ and $H_0\in \mathcal{F}^{\bot}$. Then $F\in\mathcal{F}\cap\mathcal{F}^{\bot}\subseteq {^{\bot}\mathcal{L}}$. By Lemma \ref{lem-1}, we follow that $H\in{^{\bot}\mathcal{L}}$ and then $\mathcal{F}^{\bot}\subseteq {^{\bot}\mathcal{L}}$. Dually, we have ${^{\bot}\mathcal{L}}\subseteq \mathcal{F}^{\bot}$. Since $\mathcal{F}$ is closed under kernels of deflations, it is clear that $\mathcal{F}\cap\mathcal{F}^{\bot}\subseteq{\rm Proj}(\mathcal{C})$.
\end{proof}

Combining Proposition \ref{prop-2} and Proposition \ref{prop-7}, we have the following corollary.

\begin{corollary}\label{coro-2}
    Suppose that $(\mathcal{C},\mathcal{E})$ is a WIC exact category with enough projectives and injectives. If $(\mathcal{F},\mathcal{H})$ and $(\mathcal{G},\mathcal{L})$ are complete hereditary cotorsion pairs in $\mathcal{C}$ with $\mathcal{F}\cap\mathcal{H}\subseteq\mathcal{G}$ and $\mathcal{G}\cap\mathcal{L}\subseteq\mathcal{H}$, then $\mathcal{G}=\mathcal{H}$ if and only if $(\mathcal{F},\mathcal{L})$ is an $\mathcal{E}$-admissible balanced pair in $\mathcal{C}$.
\end{corollary}

Now we focus on the exact category $(\Ch(R),\text{C-E})$. By Proposition \ref{prop-1} and Remark \ref{rmk-4}, it is a WIC exact category with enough projectives and injectives.

\begin{proposition}\label{prop-10}
    Suppose that $(\mathcal{F},\mathcal{H})$ and $(\mathcal{G},\mathcal{L})$ are complete hereditary cotorsion pairs in $\Mod R$. If $\mathcal{F}\cap\mathcal{H}\subseteq\mathcal{G}$ and $\mathcal{G}\cap\mathcal{L}\subseteq\mathcal{H}$, then $(\mathcal{F},\mathcal{L})$ is an admissible balanced pair in $\Mod R$ if and only if $(\CE(\mathcal{F}),\CE(\mathcal{L}))$ is a C-E-admissible balanced pair in $\Ch(R)$.
\end{proposition}
\begin{proof}
    If $(\mathcal{F},\mathcal{L})$ is an admissible balanced pair,
    then by \cite[Corollary 4.8]{EPZ20} one can get that $\mathcal{G}=\mathcal{H}$. Hence, the
    triple $(\text{C-E}(\mathcal{F}),\text{C-E}(\mathcal{G}),\text{C-E}(\mathcal{L}))$ is a
    complete hereditary cotorsion triple by \cite[Theorem 3.9]{YL14-1} and the assumption.
    Then $(\text{C-E}(\mathcal{F}),\text{C-E}(\mathcal{L}))$ is a C-E-admissible balanced pair
    by Proposition \ref{prop-2}.
    Conversely, suppose that $(\text{C-E}(\mathcal{F}),\text{C-E}(\mathcal{L}))$ is a
    C-E-admissible balanced pair. Note that
    $(\text{C-E}(\mathcal{F}),\text{C-E}(\mathcal{H}))$ and $(\text{C-E}(\mathcal{G}),\text{C-E}(\mathcal{L}))$
    are complete cotorsion pairs in $\Ch(R)$ by \cite[Theorem 3.9]{YL14-1}.
    Since $\mathcal{F}\cap\mathcal{H}\subseteq\mathcal{G}$ and $\mathcal{G}\cap\mathcal{L}\subseteq\mathcal{H}$,
    we have $\text{C-E}(\mathcal{F})\cap \text{C-E}(\mathcal{H})\subseteq \text{C-E}(\mathcal{G})$
    and $\text{C-E}(\mathcal{G})\cap \text{C-E}(\mathcal{L})\subseteq \text{C-E}(\mathcal{H})$
    by the definition of Cartan-Eilenberg complexes.
    Then $\text{C-E}(\mathcal{H})=\text{C-E}(\mathcal{G})$ by Corollary \ref{coro-2}.
    It follows that $\mathcal{H}=\mathcal{G}$. Indeed, let $X\in\mathcal{H}$ and view $X$ as a stalk complex in $\Ch(R)$,
    then $X\in \text{C-E}(\mathcal{H})=\text{C-E}(\mathcal{G})$.
    Again, by the definition, we get that $X\in\mathcal{G}$.
    We conclude that $(\mathcal{F},\mathcal{L})$ is an admissible balanced pair in $\Mod R$ again by \cite[Corollary 4.8]{EPZ20}.
\end{proof}

Applying this, we obtain the reverse statement of Theorem \ref{thm-1}.

\begin{theorem}\label{thm-2}
    If the cotorsion pair $({\rm GProj},{\rm GProj}^{\bot})$ is complete, then the following statements are equivalent.

    {\rm (1)} $R$ is a left virtually Gorenstein ring.

    {\rm (2)} The pair $(\CE({\rm GProj}),\CE({\rm GInj}))$ is a $\CE$-admissible balanced pair in the exact category $(\Ch(R),\CE)$.

    {\rm (3)} The pair $({\rm GProj},{\rm GInj})$ is an admissible balanced pair in $\Mod R$.
\end{theorem}
\begin{proof}
    Using Theorem \ref{thm-1} and Proposition \ref{prop-10}, we obtain the equivalences of {\rm (2)} and {\rm (3)}. The equivalence of {\rm (1)} and {\rm (3)} follows by \cite[Corollary 4.8]{EPZ20}.
\end{proof}

We give several examples when the cotorsion pair $({\rm GProj},{\rm GProj}^{\bot})$ is
complete.

\begin{corollary}\label{coro-4}
    Suppose that $R$ satisfies one of the following conditions.

    {\rm (1)} The ring $R$ is left n-perfect and right coherent. In particular, this includes
    commutative Noetherian rings of finite Krull dimension;

    {\rm (2)} There exists a supercompact cardinal $\geq {\rm Card}R$;

    {\rm (3)} Vop$\check{e}$nka's principle holds in the sense of \cite{C25}.

    Then the statements in Theorem \ref{thm-2} hold true.
\end{corollary}
\begin{proof}
    For {\rm (1)}, it follows from \cite[Proposition 6]{EIO17}.

    For {\rm (2)} and {\rm (3)}, it follows from \cite[Theorem 5.2 and Proposition 5.6]{CS25} and \cite[Theorem 1.4(B)]{C25}.
\end{proof}

Let $\mathcal{PGF}$ be the class of all projectively coresolved Gorenstein flat modules (see the definition in \cite[Section 4]{SS20}). Note that ${\rm GProj}\subseteq{\rm GFlat}$ is equivalent to $\mathcal{PGF}={\rm GProj}$ by \cite[Theorem 3]{I20}, and then in this case the cotorsion pair $({\rm GProj},{\rm GProj}^{\bot})$ is cogenerated by a set by \cite[Theorem 4.9]{SS20} and hence it is complete. This also provides an approach to obtaining the complete cotorsion pair $({\rm GProj},{\rm GProj}^{\bot})$. By \cite[Theorem 4]{I20}, one obtains $\mathcal{PGF} = {\rm GProj}$ in case (1) of Corollary \ref{coro-4}.

\section{C-E version of Tate cohomology}
In this section, we study C-E versions of Tate cohomology. For simplicity, we supress the dot in the notation for complexes. We recall from \cite{E11} that a complex
$G\in\Ch(R)$ is said to be {\it $\CE$ Gorenstein projective} if there exists a $\CE$-exact
sequence of complexes
$$ \cdots\to P_{1}\to P_{0}\to P^1\to P^2\to\cdots$$
such that $G=\Im(P_0\to P^1)$ with each $P_i$ C-E projective complex, which remains exact after applying $\Hom(-,P)$ for any arbitrary C-E projective complex $P$.
By \cite[Proposition 2.2]{YL14}, a complex $G$ is C-E Gorenstein projective if and only if $B(G)$ and $H(G)$ are in $\Ch(R\text{-}{\rm GProj})$.

\begin{definition}\cite[Definition 2.12]{YL14}\label{def-3}
    A complex $G$ is said to have {\it $\CE$ Gorenstein projective dimension at most $n$} if there exists a $\CE$-exact sequence $0\to K_n\to P_{n-1}\to \cdots\to P_1\to P_0\to G\to 0$ with each $P_i$ a C-E projective complex and $K_n$ $\CE$ Gorenstein projective complex, denoted by $\text{\rm C-E Gpd}G\leq n$. If $n$ is the least, then we set $\text{\rm C-E Gpd}G=n$ and we set $\text{\rm C-E Gpd}G= \infty$ if there is no such $n$.
\end{definition}
We have the following characterizations of finite C-E Gorenstein projective dimension.

\begin{proposition}\label{prop-3}
    For a complex $G\in\Ch(R)$, the following statements are equivalent.

    {\rm (1)} $\text{\rm C-E Gpd}G\leq n$.

    {\rm (2)} $B(G)$ and $H(G)$ are in $\Ch(R\text{-}{\rm GProj}_n)$.

    {\rm (3)} $C(G)$ and $G$ are in $\Ch(R\text{-}{\rm GProj}_n)$.
\end{proposition}

\begin{proof}
    ``{\rm (1)}$\Rightarrow${\rm (2)} and {\rm (3)}'' By definition, there exists a $\CE$-exact sequence $\Sigma$: $0\to K_n\to P_{n-1}\to \cdots\to P_1\to P_0\to G\to 0$ with each $P_i$ a C-E projective complex and $K_n$ C-E Gorenstein projective complex. Thus $B(\Sigma)$, $H(\Sigma)$ and $C(\Sigma)=\Sigma/B(\Sigma)$ are all exact and {\rm (2)} and {\rm (3)} are clear by \cite[Proposition 2.2]{YL14} and \cite[Lemma 2.1]{Z13}.

    ``{\rm (2)}$\Rightarrow${\rm (1)}'' By assumption, we know that ${\rm Gpd}B_n(G)\leq n$ and ${\rm Gpd}H_n(G)\leq n$ for each $n\in\mathbb{Z}$. Since $0\to B_{m+1}(G)\to Z_m(G)\to H_m(G)\to 0$ and $0\to Z_m(G)\to G_m\to B_m(G)\to 0$ are exact, the Horseshoe lemma allows one to construct a partial projective resolution of $G_m$
    $$ 0\to K_n^{G_m}\to P_{n-1}^{B_{m+1}}\oplus P_{n-1}^{B_m}\oplus P_{n-1}^{H_m}\to\cdots\to P_{0}^{B_{m+1}}\oplus P_{0}^{B_m}\oplus P_{0}^{H_m}\to G_m\to 0$$
    with all $P$-terms projective. Let $P^i$ be the complex with each degree $P_{i}^{B_{m+1}}\oplus P_{i}^{B_m}\oplus P_{i}^{H_m}$ and the differential is $\left(\smqty{0&1&0\\0&0&0\\0&0&0}\right)$, then $P^i$ is C-E projective. So we get the desired exact sequence $0\to K \to P^{i-1}\to \cdots\to P^0\to G\to 0$. We focus on the short exact sequence $0\to K^1\to P^0\to G\to 0$. We get the following commutative diagram $$\xymatrix{
       0 \ar[r] & K_1^{G_{m+1}} \ar@/_2pc/[ddd] \ar[r] \ar@{->>}[d] & P_{0}^{B_{m+2}}\oplus P_{0}^{B_{m+1}}\oplus P_{0}^{H_{m+1}} \ar[r] \ar@{->>}[d] & G_{m+1} \ar[r] \ar@{->>}[d] & 0\\
       0 \ar[r] & K_1^{B_{m+1}} \ar[r] \ar@{>->}[d] & P_{0}^{B_{m+1}} \ar[r] \ar@{>->}[d] & B_{m+1}(G) \ar[r] \ar@{>->}[d] & 0 \\
       0 \ar[r] & K_1^{Z_{m}} \ar[r] \ar@{>->}[d] & P_{0}^{B_{m+1}}\oplus P_{0}^{H_{m}} \ar[r] \ar@{>->}[d] & Z_{m}(G) \ar[r] \ar@{>->}[d] & 0\\
       0 \ar[r] & K_1^{G_{m}} \ar[r]  & P_{0}^{B_{m+1}}\oplus P_{0}^{B_{m}}\oplus P_{0}^{H_{m}} \ar[r]  & G_{m} \ar[r]  & 0
    }$$ Note that $\Im (K_1^{G_{m+1}}\to K_1^{G_{m}})=\Im(K_1^{B_{m+1}}\rightarrowtail  K_1^{G_{m}})=K_1^{B_{m+1}}$ and $\Ker(K_1^{G_{m+1}}\to K_1^{G_{m}})=\Ker(K_1^{G_{m}}\twoheadrightarrow K_1^{B_{m}})=K_1^{Z_m}$, it follows that $K^1$ is in $\Ch(R\text{-}{\rm GProj}_1)$ and $0\to K^1\to P^0\to G\to 0$ is $\CE$-exact. One then proves that the desired sequence is $\CE$-exact and $K$ is C-E Gorenstein projective by induction.

    ``{\rm (3)}$\Rightarrow${\rm (2)}'' Note that we have the following exact sequences $0\to B_{n+1}(G)\to G_{n}\to C_{n+1}(G)\to 0$ and $0\to H_n(G)\to G_n/B_{n+1}(G)\to B_{n}(G)\to 0$, then the statement holds by \cite[Theorem 1.1]{H22}.
\end{proof}

\begin{definition}\label{def-1}
    A complex $G$ {\it has a complete $\CE$ projective resolution} if there exists a diagram $T\stackrel{u}\to P\to G$ with $T$ a $\CE$-exact complex of C-E projective complexes that remains exact after applying the functor $\Hom(-,P')$ for any C-E projective complex $P'$, $P\to G$ a $\CE$-exact $\text{C-E}(\mathcal{P})$-resolution, and $u$ is a morphism of complexes of complexes such that $u_n:T_n\to P_n$ is an isomorphism for all $n\gg 0$.
\end{definition}

\begin{proposition}\label{prop-5}
    A complex $G$ has finite $\CE$ Gorenstein projective dimension if and only if $G$ has a complete $\CE$ projective resolution.
\end{proposition}
\begin{proof}
    If the complex $G$ has finite C-E Gorenstein projective dimension, then there exists a $\CE$-exact sequence $0\to K_n\to P_{n-1}\to \cdots\to P_1\to P_0\to G\to 0$ with $P_i$ C-E projective and $K_n$ C-E Gorenstein projective. Hence, there exists a C-E-exact sequence of C-E projective complexes
    $$Q:\cdots\to Q_{1}\to Q_{0}\to Q^{1}\to\cdots$$ which remains exact after applying $\Hom(-,P')$ for every C-E projective complex $P'$, and with
    with $K_n=\Im(Q_{0}\to Q^1)$. Now let $T:=Q$ and
    $$P:=\cdots\to Q_{0}\to P_{n-1}\to \cdots\to P_1\to P_0\to 0\to \cdots,$$
    we get the complete projective resolution. Conversely, if $G$ has a complete C-E projective resolution, then there exists a diagram $T\stackrel{u}\to P\to G$ as in Definition \ref{def-1}, so $G$ has finite C-E Gorenstein projective dimension.
\end{proof}
Combining Propositions \ref{prop-3} and \ref{prop-5}, we have the following corollary.
\begin{corollary}
    For a complex $G\in\Ch(R)$, the following statements are equivalent.

    {\rm (1)} $G$ has finite $\CE$ Gorenstein projective dimension.

    {\rm (2)} $G$ has a complete $\CE$ projective resolution.

    {\rm (3)} $B(G)$ and $H(G)$ are in $\Ch(R\text{-}{\rm GProj}_{\infty})$.

    {\rm (4)} $C(G)$ and $G$ are in $\Ch(R\text{-}{\rm GProj}_{\infty})$.
\end{corollary}

In \cite{EEI12}, the authors give some properties of bicomplexes in the category of modules. We give their abelian version. Let $\mathcal{A}$ be an abelian category and $(X,d)$ be a complex in $\mathcal{A}$, we let $Z(X)$ be $\Ker(d)$, $B(X)$ be $\Im(d)$ and $H(X)$ be the quotient complex $Z(X)/B(X)$. By a double complex $X$ in $\mathcal{A}$, we mean a bigraded object $(X^{(i,j)})_{(i,j)\in \mathbb{Z}\times \mathbb{Z}}$ along with $d'$ and $d''$ of bidegrees $(-1,0)$ and $(0,-1)$, respectively, such that $d'^2=0$, $d''^2=0$ and $d'\circ d''+d''\circ d'=0$. Similarly, we say $X=(X^{(i,j)})_{(i,j)\in \mathbb{Z}\times \mathbb{Z}}$ is a {\it bicomplex} if we take the axioms for a double complex and replace the condition $d'\circ d''+d''\circ d'=0$ with the condition $d'\circ d''=d''\circ d'$. Given the bicomplex $X$, we let $Z'(X)$ be the bicomplex $\Ker(d')$, $B'(X)$ be $\Im(d')$ and $H'(X)$ be the quotient complex $Z'(X)/B'(X)$. We note that these complexes have their $d'=0$. Dually, one can define $Z''(X)$, $B''(X)$ and $H''(X)$.

\begin{proposition}\label{prop-4}
    Let $X$ be a bicomplex such that $H'(X)=H''(X)=0$, i.e. such that $X$ has exact rows and columns. Then $H'(Z''(X))=H''(Z'(X))$. In this case, we set $H(X)=H'(Z''(X))=H''(Z'(X))$.
\end{proposition}
\begin{proof}
    We set $Z'(i,j)=\Ker(X^{(i,j)}\to X^{(i-1,j)})$ and $B'(i,j)=\Im(X^{(i,j)}\to X^{(i-1,j)})$ for each $(i,j)\in\mathbb{Z}\times\mathbb{Z}$. Dually we set $Z''(i,j)$ and $B''(i,j)$. Then we note that \begin{align*}
    H'(Z''(X))^{(i,j)}=\Ker(Z''(i,j)\to Z''(i-1,j))/\Im(Z''(i+1,j)\to Z''(i,j)),\\
    H''(Z'(X))^{(i,j)}=\Ker(Z'(i,j)\to Z'(i,j-1))/\Im(Z'(i,j+1)\to Z'(i,j)).
    \end{align*}

First, we show that $\Ker(Z''(i,j)\to Z''(i-1,j))$ and $\Ker(Z'(i,j)\to Z'(i,j-1))$ are both the pullbacks of $$\xymatrix{
   & Z'(i,j) \ar@{>->}[d] \\
    Z''(i,j) \ar@{>->}[r] & X^{(i,j)}.
}$$  Fix $j$, we let $Z''(X)_j=\cdots\to Z''(i+1,j)\to Z''(i,j)\to Z''(i-1,j)\to\cdots$ and  $X_j= \cdots\to X^{(i+1,j)}\to X^{(i,j)}\to X^{(i-1,j)}\to \cdots$ be the complexes.
Note that the chain map $Z''(X)_j\rightarrowtail X_j$ induces the following diagram $$\xymatrix@R=1cm{
    \mathop{\Ker(Z''(i,j)\to Z''(i-1,j))\,}\limits_{\,} \ar@{>->}[r] \ar@{>-->}[d]^{f} & \mathop{Z''(i,j)}\limits_{\,} \ar@{>->}[d]\\
    Z'(i,j) \ar@{>->}[r] & X^{(i,j)}.
}$$ The morphism $f$ is a monomorphism by the Snake Lemma. Now if we have the following diagram
$$\xymatrix@R=1.2cm{
    X \ar@/^2pc/[rrd]^{g} \ar@/_2pc/[rdd]_{h} \ar@{-->}[rd]^{\psi} & & \\
     & \mathop{\Ker(Z''(i,j)\to Z''(i-1,j))\,}\limits_{\,} \ar@{>->}[r]^<<<<{\iota_1} \ar@{>-->}[d]^{f} & \mathop{Z''(i,j)}\limits_{\,}
     \ar[r] \ar@{>->}[d] & Z''(i-1,j) \ar[d]\\
     & Z'(i,j)\, \ar@{>->}[r] & X^{(i,j)} \ar[r] & X^{(i-1,j)}.
}$$By chasing the diagram, there uniquely exists $\psi:X\to \Ker(Z''(i,j)\to Z''(i-1,j))$ such that $\iota_1\psi=g$, and it follows that $f\psi=h$. The argument is similar for $\Ker(Z'(i,j)\to Z'(i,j-1))$. Therefore, they are isomorphic and we use $Z'''$ to denote them. For simplicity, we set $B''=\Im(Z''(i+1,j)\to Z''(i,j))$ and $B'=\Im(Z'(i,j+1)\to Z'(i,j))$. It suffices to show that $B'\cong B''$. Since $X_{j+1}$ is exact, then the chain map $X_{j+1}\twoheadrightarrow Z''(X)_j$ induces an epimorphism $\phi:Z'(i,j-1)\twoheadrightarrow B''$. By chasing diagram, one can verify that $l\phi=d''$. Then there uniquely exists $f':B''\to B'$ such that the following diagram commutes:
$$\xymatrix{
    & & Z'(i,j+1) \ar@/^2pc/[dd]^{d''} \ar@{-->>}[d]^{\phi} & & \\
    & & \mathop{B''}\limits_{\,} \ar@{>->}[d]^{l} \ar@{-->}@[blue][ld]_{f'} & & \\
    0 \ar[r] & B' \ar@{>->}[r]^{\iota'} & Z'(i,j) \ar[r] & \Coker d'' \ar[r] & 0.
}$$
Similarly, in the case of $B'$, there uniquely exists $g':B'\to B''$ and we get the following commutative diagram $$\xymatrix{
         & \mathop{B''}\limits_{\,} \ar@{>->}[d]^{i} \ar@{>->}[rd]^{i'} \ar@{-->}@[blue]@<0.5ex>[ld]^{\textcolor{blue}{f'}} & \\
         B'\, \ar@{-->}@[blue]@<0.5ex>[ru]^{\textcolor{blue}{g'}} \ar@{>->}[r]_{\iota} \ar@{>->}[rd]_{\iota'} & \mathop{Z'''\,}\limits_{\,} \ar@{>->}[d]^{i_1} \ar@{>->}[r]^{\iota_1} & \mathop{Z''(i,j)}\limits_{\,} \ar@{>->}[d]^{i_2}\\
          & Z'(i,j)\, \ar@{>->}[r]^{\iota_2} & X^{(i,j)}
}$$ such that $\iota_1ig'=i'g'=\iota_1\iota $ and $i_1\iota f'=\iota'f'=i_1i$. Therefore,
$B'\cong B''$ since each solid arrow is monomorphic.
\end{proof}

Now let $C=(C_i)$ and $D=(D_j)$ be complexes with each $C_i$ and $D_j$ in $\mathcal{A}$. We construct a bicomplex $\Hom(C,D)$ with $\Hom(C,D)^{(i,j)}=\Hom(C_i,D_j)$ and $d'=\Hom(d_C,D)$, $d''=\Hom(C,d_D)$. The proofs of the following results are similar to \cite[Proposition 3.1-3.2, Theorem 3.3 and Corollary 3.4]{EEI12}. We state them without proof.

\begin{proposition}
    The following statements hold.

    {\rm (1)} If $C$ is an exact complex and $D$ is any complex, then $Z'(\Hom(C,D))\cong \Hom(Z(C),D)$ where the isomorphism is an isomorphism of bicomplexes.

    {\rm (2)} If $C$ is any complex and $D$ is an exact complex, then $Z''(\Hom(C,D))\cong \Hom(C,Z(D))$ where the isomorphism is an isomorphism of bicomplexes.
\end{proposition}

\begin{proposition}\label{prop-11}
    If $C$ and $D$ are both exact complexes and if $\Hom(C,D)$ has exact rows and columns, then for each $(i,j)\in\mathbb{Z}\times\mathbb{Z}$, we have $$H(\Hom(C,D))^{(i,j)}\cong H^j(\Hom(Z_i(C),D))\cong H^i(\Hom(C,Z_j(D))).$$
\end{proposition}
\begin{proof}
    This follows from Proposition \ref{prop-4}.
\end{proof}

\begin{corollary}\label{coro-1}
    Under the same assumptions as in Proposition \ref{prop-11}, for any $n\in\mathbb{Z}$ we have $H^n(\Hom(Z_0(C),D))\cong H^n(\Hom(C,Z_0(D)))$.
\end{corollary}

Note that $\Ext^n_{\text{C-E}}(M,N)\cong H^n(\Hom(P,N))$ with any C-E projective resolution of $M$ by \cite[Appendix B]{G16} and \cite[Chapter XII.4]{M63}. This means that the complete C-E projective resolution is unique up to homotopy.

\begin{definition}
    Let $M$ be a complex that has a complete C-E projective resolution $T\to P\to M$. For each complex $N$ and each integer $n\in\mathbb{Z}$, the {\it Tate cohomology group relative to $\CE$}, denoted $\widehat{\Ext}^n_{\text{C-E}}(M,N)$, is defined by $\widehat{\Ext}^n_{\text{C-E}}(M,N)=H^n(\Hom(T,N))$. Dually, let $N$ be any complex that has a complete C-E injective resolution $N\to E\to I$, for any complex $M$ and $n\in\mathbb{Z}$, set $\overline{\Ext}^n_{\text{C-E}}(M,N)=H^n(\Hom(M,I))$.
\end{definition}

\begin{theorem}\label{thm-4}
    If $T\to P\to M$ and $N\to I\to U$ are complete $\CE$ projective and injective resolutions of $M$ and $N$, respectively, then the cohomology groups of $\Hom(T,N)$ and $\Hom(M,U)$ are naturally isomorphic.
\end{theorem}
\begin{proof}
    By Proposition \ref{prop-5} and its dual, we set $\text{\rm C-E Gpd}M=d$ and $\text{\rm C-E Gid}N=l$. Without loss of generality we may assume that $l\geq d$. Then there exists a partial C-E projective resolution $0\to G\to P_{d-1}\to \cdots\to P_0\to M$  of $M$. Then $G$ is C-E Gorenstein projective and $M$ has a complete C-E projective resolution $T\to P\to M$ with $T=\cdots\to T_{d}\to T_{d-1}\to \cdots$ a $\Hom(-,\text{C-E}(\mathcal{P}))$ $\CE$-exact complex of C-E projective complexes such that $G=\Im (T_d\to T_{d-1})=\Ker (T_{d-1}\to T_{d-2})$ and $P= \cdots\to T_{d} \to P_{d-1}\to \cdots\to P_0\to 0 \to\cdots$. Let $G_0=\Ker(P_0\to M)$ and $\text{\rm C-E Gpd}G_0=d-1$. Then $G_0$ has a complete C-E projective resolution $T[-1]\to \bar{P}\to G_0$ with $\bar{P}=\cdots\to T_{d} \to P_{d-1}\to \cdots\to P_{1}\to 0 \to\cdots$. Dimension shifting gives $\widehat{\Ext}^j_{\text{C-E}}(M,-)\cong \widehat{\Ext}^{j-1}_{\text{C-E}}(G_0,-)$, so $\widehat{\Ext}^j_{\text{C-E}}(M,-)\cong \widehat{\Ext}^{j-d}_{\text{C-E}}(G,-)$. Suppose that $0\to N\to E^0\to E^1\to \cdots \to E^{l-1}\to H\to 0$ is a partial C-E injective resolution of $N$ and $H$ is C-E Gorenstein injective.

    Now let $H_0=\Ker(E^1\to E^2)$ and we have a $\CE$-exact sequence $0\to N\to E^0\to H_0\to 0$. For each $i\in\mathbb{Z}$, we have a short exact sequence $0\to \Hom(T_i,N)\to \Hom(T_i,E^0)\to \Hom(T_i,H_0)\to 0$ and then get $0\to\Hom(T,N)\to \Hom(T,E^0)\to \Hom(T,H_0)\to 0$. Note that $E^0$ is C-E injective, so $\widehat{\Ext}^i_{\text{C-E}}(M,H_0)\cong \widehat{\Ext}^{i+1}_{\text{C-E}}(M,N)$. Then we have $\widehat{\Ext}^i_{\text{C-E}}(M,N)\cong \widehat{\Ext}^{i+1}_{\text{C-E}}(M,H_0)\cong \cdots\cong\widehat{\Ext}^{i+l}_{\text{C-E}}(M,H)$. Combining these isomorphisms, we obtain that $\widehat{\Ext}^i_{\text{C-E}}(M,N)\cong \widehat{\Ext}^{i-d}_{\text{C-E}}(G,N)\cong \widehat{\Ext}^{i-d+l}_{\text{C-E}}(G,H)$.

    Similarly, one can get $\overline{\Ext}^i_{\text{C-E}}(M,N)\cong \overline{\Ext}^{i+l}_{\text{C-E}}(M,H)\cong \overline{\Ext}^{i+l-d}_{\text{C-E}}(G,H)$. Since $G$ has a complete
    C-E projective resolution $T[-d]\to P_G \to G$ with $P_G=\tau_{>0}T[-d]$ and $H$ has a
    complete C-E injective resolution $H\to I_H\to U[l]$ with $I_H=\tau_{<0}U[l]$. Thus, by
    Corollary \ref{coro-1}, we have $\widehat{\Ext}^{n}_{\text{C-E}}(G,H)=H^n(\Hom(T[-d],Z_0(U[l])))\cong H^n(\Hom(Z_0(T[-d]),U[l]))=\overline{\Ext}^{n}_{\text{C-E}}(G,H)$ for each
    $n\in\mathbb{Z}$.
\end{proof}


\begin{remark} \label{rmk-2}
    Suppose that $R$ is left virtually Gorenstein. As already observed in the proof of
    Theorem \ref{thm-1}, the pair $({\rm GProj},{\rm GProj}^{\perp})$ is a complete hereditary
    cotorsion pair, and the pair
    $$(\text{C-E}({\rm GProj}),\text{C-E}({\rm GProj}^{\perp}))$$
    is also a complete hereditary cotorsion pair in $(\Ch(R),\text{C-E})$. Hence the triple
    $$\mathfrak{M}=(\text{C-E}({\rm GProj}),\text{C-E}({\rm GProj}^{\perp}),\Ch(R))$$
    defines an exact model structure on $(\Ch(R),\text{C-E})$. Using the same notation as in the
    proof of Theorem \ref{thm-4}, we have that
    $\widehat{\Ext}^i_{\text{C-E}}(M,N)\cong \widehat{\Ext}^{i-d}_{\text{C-E}}(G,N)\cong \widehat{\Ext}^{i-d+l}_{\text{C-E}}(G,H)$.
    In \cite{G21}, Gillespie shows how exact model structures can be used to obtain canonical resolutions and to define new
    relative derived functors from them. For the model structure $\mathfrak{M} $ above, the construction in \cite[Section 1]{G21} yields that, for the C-E Gorenstein projective complex $G$, the object $W_G$ described in \cite[Section 1]{G21} is precisely the $\CE$-exact complex $T$. Therefore, \cite[Definition 4.6]{G21} gives $\ell\Ext_{\mathfrak{M}}^n(G,H)=\widehat{\Ext}^n_{\text{C-E}}(G,H)$.
    Similarly, using the exact model structure $$\mathfrak{M}'=(\Ch(R),\text{C-E}(^{\perp}{\rm GInj}),\text{C-E}({\rm GInj})), $$ we obtain
    $r\Ext_{\mathfrak{M}'}^n(G,H)=\overline{\Ext}^n_{\text{C-E}}(G,H)$.
    (Note that, as mentioned in \cite[1A.]{G21}, the results in \cite{G21} hold for the case of exact
    categories.)
    \newline Now, since \(R\) is left virtually Gorenstein, one has ${\rm GProj}^{\perp}={}^{\perp}{\rm GInj},$
and therefore
\[
\text{\rm C-E}({\rm GProj}^{\perp})
=
\text{\rm C-E}({}^{\perp}{\rm GInj}).
\]
Together with \cite[Lemma 5.4]{EG19} and \cite[Lemma 8.26]{G25}, this shows that
the two relevant model structures $\mathfrak{M}$ and $\mathfrak{M}'$ are Quillen equivalent. Thus
\cite[Theorem 6.3]{G21} gives natural isomorphisms
\[
\overline{\Ext}^{\,n}_{\text{\rm C-E}}(G,H)
\cong
\widehat{\Ext}^{\,n}_{\text{\rm C-E}}(G,H).
\]
Consequently, for the original complexes \(M\) and \(N\), we recover Theorem \ref{thm-4},
\[
\overline{\Ext}^{\,n}_{\text{\rm C-E}}(M,N)
\cong
\widehat{\Ext}^{\,n}_{\text{\rm C-E}}(M,N).
\]
The point to emphasize is that Gillespie's construction is categorical in
nature and depends on the completeness of the cotorsion pairs involved. Our
definition, however, is based directly on complete Cartan--Eilenberg
resolutions. Thus, it does not require the completeness of the cotorsion pair
\(({\rm GProj},{\rm GProj}^{\perp})\) as an a priori assumption.

\end{remark}

As mentioned above, $\Ext^n_{\text{C-E}}(M,N)\cong H^n(\Hom(P,N))$ with any $\CE$-exact C-E projective resolution $P$ of $M$. Hence we consider the relative derived functors in $(\Ch (R), \text{C-E})$.

\begin{definition}
    Let $\mathcal{D}$ be a class of complexes and $M$ be a complex that has a $\CE$-exact $\mathcal{D}$-resolution $P:\cdots\to P_1\to P_0\to M\to 0$. Then we define $\Ext^n_{\CE,\mathcal{D}}(M,N)=H^n(\Hom(P_{\bullet},N))$ for any complex $N$ and $n\geq 0$, where $P_{\bullet}$ is its deleted resolution.
\end{definition}

If we take $\mathcal{D}$ to be the class of all C-E projective complexes, then $\Ext^n_{\text{C-E},\text{C-E}(\mathcal{P})}(M,N)=\Ext^n_{\text{C-E}}(M,N)$. Now let $\mathcal{D}\subseteq\mathcal{C}$ be two classes in $(\Ch (R), \text{C-E})$. For a complex $M$ that has both a $\CE$-exact $\mathcal{D}$-resolution $P:\cdots\to P_1\to P_0\to M\to 0$ and a $\CE$-exact $\mathcal{C}$-resolution $C:\cdots\to C_1\to C_0\to M\to 0$, there exists a chain map $u=(u_n)_{n\geq 0}:P\to C$ such that the following diagram commutes
$$\xymatrix{
   P: \ar[d]^{u} & \cdots \ar[r] & P_1 \ar[r]^{f_1} \ar[d]^{u_1} & P_0 \ar[r] \ar[d]^{u_0} & M \ar[r] \ar@{=}[d] & 0 \\
   C: & \cdots\ar[r] & C_1 \ar[r]^{g_1} & C_0 \ar[r] & M \ar[r] & 0.
}$$
Define $M(u)$ to be the mapping cone of the chain map
$u:P_{\bullet}\to C_{\bullet}$ with the following form
$$M(u): \cdots \to C_2\oplus P_1\to C_1\oplus P_0\to C_0\to 0.$$
The differential map is of the form $\left(\smqty{g&u\\0&-f}\right)$. Note that $P$ and $C$ are $\CE$-exact, then we can get commutative diagrams in $\Mod R$ after applying $Z_n, H_n, B_n$ and $(\cdot)_{n}$ to the chain map $u$ and so $M(u)$ is $\CE$-exact. Then we define $\widehat{\Ext}^n_{\text{C-E},\mathcal{C},\mathcal{D}}(M,N)=H^{n+1}(\Hom(M(u),N))$ to be the {\it generalized Tate cohomology with respect to C-E}. This definition is well-defined, the proof is similar to that in \cite{I05}. If we take $\mathcal{C}=\text{C-E}({\rm GProj}):=\mathcal{G}$ and $\mathcal{D}=\text{C-E}(\mathcal{P})$, then we obtain the connection between Tate cohomology and generalized Tate cohomology under certain conditions.

\begin{proposition}\label{prop-6}
    Let $M$ be a complex with $\text{\rm C-E Gpd}M<\infty$, then for each complex $N$, we have $\widehat{\Ext}^n_{\CE,\mathcal{G},\CE(\mathcal{P})}(M,N)\cong \widehat{\Ext}^n_{\CE}(M,N)$ for any $n\geq 1$.
\end{proposition}
\begin{proof}
    Let $\text{\rm C-E Gpd}M=n$, then there exists a $\CE$-exact sequence $$0\to C\to P_{n-1}\stackrel{f_{n-1}}\to P_{n-2}\to \cdots \to P_{1}\stackrel{f_{1}}\to P_0\stackrel{f_{0}}\to M\to 0$$ with each $P_{i}$ a C-E projective and $C$ a C-E Gorenstein projective. Therefore, we get a complete C-E projective resolution of $M$: $T\stackrel{u}\to P\to M$: $$\xymatrix{
        T: \ar[d]^{u} & \cdots\ar[r] &P^{2} \ar[r]^{d^{2}} \ar@{=}[d] & P^{1}\ar[r]^{d^{1}} \ar@{=}[d] & P^{0}\ar[r]^{d^{0}} \ar[d]^{u_0} & \cdots\ar[r] & P^{-n+1}\ar[r]^{d^{-n+1}} \ar[d] & P^{-n} \ar[r] \ar[d] & \cdots\\
        P: & \cdots\ar[r] & P^{2}\ar[r] & P^{1}\ar[r]^{u_0d^{1}} & P_{n-1}\ar[r]^{f_{n-1}} & \cdots\ar[r] & P_{0}\ar[r] & 0\ar[r] & \cdots,
    }$$
    so there exists a $\CE$-exact commutative diagram $$\xymatrix{
        0\ar[r] & C \ar[r] \ar@{=}[d] & P^0 \ar[r] \ar[d] & P^{-1} \ar[r] \ar[d] & \cdots \ar[r] & P^{-n+2} \ar[r] \ar[d] & P^{-n+1} \ar[r] \ar[d] & D \ar[d] \ar[r] & 0\\
        0\ar[r] & C \ar[r] & P_{n-1} \ar[r] & P_{n-2} \ar[r] & \cdots \ar[r] & P_{1} \ar[r] & P_{0} \ar[r] & M \ar[r] & 0\\
    }$$
with $D$ being C-E Gorenstein projective. Then similarly we can get a $\CE$-exact sequence $G_M: 0\to P^0\to P_{n-1}\oplus P^{-1}\to \cdots \to P_{1}\oplus P^{-n+1}\to P_0\oplus D \stackrel{\beta}\to M \to 0$ whose differential map is of the form $\left(\smqty{f&u\\0&-d}\right)$. We show that the above sequence is a $\CE$-exact $\mathcal{G}$-resolution of $M$. Since $\text{C-E pd}\Ker \beta<\infty$, then for any C-E Gorenstein projective complex $L$, we have $\Ext^1_{\text{C-E}}(L,\Ker\beta)=0$ by \cite[Proposition 2.14]{YL14}, so $\beta$ is a $\mathcal{G}$-precover of $M$. Repeating this argument, we obtain the desired result. Then we construct a chain map $e:P_M\to G_M$: $$\xymatrix@C=1.3em{
    P_M:\ar[d]^{e} & \cdots\ar[r] & P^{2}\ar[d]^{0} \ar[r] & P^{1} \ar[r] \ar[d]^{d^{1}}& P_{n-1} \ar[r]\ar[d]^{(\smqty{1\\0})} & \cdots\ar[r] & P_0  \ar[r]\ar[d]^{(\smqty{1\\0})} & M \ar[r]\ar@{=}[d] & 0 \ar[r] & \cdots\\
    G_M: & \cdots\ar[r] & 0 \ar[r] & P^{0} \ar[r] & P_{n-1}\oplus P^{-1} \ar[r] & \cdots\ar[r] & P_0\oplus D \ar[r] & M \ar[r] & 0 \ar[r] & \cdots.
}$$
We get its mapping cone $M(e): \cdots\to P^{2}\to P^{1}\to P^0\oplus P^{n-1}\to P_{n-1}\oplus P^{-1}\oplus P^{n-2}\to \cdots \to P_{1}\oplus P^{-n+1}\oplus P_0\to P_0\oplus D\to 0\to\cdots$ with the differential map being of the form $\left(\smqty{f&u&1\\0&-d&0\\0&0&-f}\right)$.
Then $\widehat{\Ext}^n_{\text{C-E},\mathcal{G},\text{C-E}(\mathcal{P})}(M,N)=H^{n+1}(\Hom(M(e),N))$ and $\widehat{\Ext}^n_{\text{C-E}}(M,N)=H^n(\Hom(T,N))=H^{n+1}(\Hom(T[1],N))$. One verifies that there exists a homotopy equivalence between $M(e)$ and $T[1]$ by using the same chain maps as in \cite{I05}. Hence, $\widehat{\Ext}^n_{\text{C-E},\mathcal{G},\text{C-E}(\mathcal{P})}(M,N)\cong\widehat{\Ext}^n_{\text{C-E}}(M,N)$.
\end{proof}

Then we get the C-E version of Avramov-Martsinkovsky's exact sequence.

\begin{corollary}\label{coro-3}
    Let $M$ be a complex with $\text{\rm C-E Gpd}M=n<\infty$. For each complex $N$ there is an exact sequence: $0\to \Ext^1_{\CE,\mathcal{G}}(M,N)\to \Ext^1_{\CE}(M,N)\to \widehat{\Ext}^1_{\CE}(M,N)\to \cdots \to \Ext^i_{\CE,\mathcal{G}}(M,N)\to \Ext^i_{\CE}(M,N)\to \widehat{\Ext}^i_{\CE}(M,N)\to\cdots\to \Ext^n_{\CE}(M,N)\to \widehat{\Ext}^n_{\CE}$ $(M,N)\to 0$.
\end{corollary}
\begin{proof}
    The existence of $\widehat{\Ext}^n_{\text{C-E},\mathcal{G},\text{C-E}(\mathcal{P})}(M,N)$ is clear. Now we have a short exact sequence of complexes $0\to G_{\bullet}\to M(u)\to P_{\bullet}[1]\to 0$ with each degree split exact. So for any complex $N$, we have $0\to \Hom(P_{\bullet}[1],N)\to \Hom(M(u),N)\to \Hom(G_{\bullet},N)\to 0$, then it induces a long exact sequence \begin{align*}
        \cdots \to H^n(\Hom(P_{\bullet}[1],N))&\to H^n(\Hom(M(u),N))\to H^n(\Hom(G_{\bullet},N))\\
        \to H^{n+1}(\Hom(P_{\bullet}[1],N))\to \cdots. &
    \end{align*}
    One obtains the following long exact sequence $$0\to H^1(\Hom(G_{\bullet},N))\to H^2(\Hom(P_{\bullet}[1],N))\to H^2(\Hom(M(u),N))\to \cdots.$$ Combining with Proposition \ref{prop-6}, we get the desired sequence.
\end{proof}

\begin{definition}\label{def-2}
    Let $M$ be a complex of right $R$-modules that has a $\CE$-exact C-E Gorenstein projective resolution $G$ of $M$. For each complex $N$, define $\Tor^{\CE,\mathcal{G}}_n(M,N)=H_n(G_{\bullet}\otimes N)$ for any $n\geq 0$.
\end{definition}

\begin{definition}
    Let $M$ be a complex of right $R$-modules that has a complete C-E projective resolution $T\stackrel{u}\to P\to M$. For any complex $N$, we define $\widehat{\Tor}_n^{\text{C-E}}(M,N)=H_n(T\otimes N)$ for any $n\in\mathbb{Z}$.
\end{definition}

Let $M$ be a complex that has a $\CE$-exact $\mathcal{G}$-resolution $G$ of $M$. Let $P$ be a $\CE$-exact $\text{C-E}(\mathcal{P})$-resolution of $M$, $u:P_{\bullet}\to G_{\bullet}$ be a map induced by $\Id_M$ and $M(u)$ denote its mapping cone. For any complex $N$, we define $\widehat{\Tor}_n^{\text{C-E},\mathcal{G},\text{C-E}(\mathcal{P})}(M,N)=H_{n+1}(M(u)\otimes N)$.

\begin{proposition}\label{prop-8}
    If $\text{\rm C-E Gpd}M<\infty$, then $\widehat{\Tor}_n^{\CE,\mathcal{G},\CE(\mathcal{P})}(M,N)\cong \widehat{\Tor}_n^{\CE}(M,N)$ for any $n\geq 1$.
\end{proposition}
\begin{proof}
    Let $\text{\rm C-E Gpd}M=n$, then there is an exact sequence $$0\to C\to P_{n-1}\stackrel{f_{n-1}}\to P_{n-2}\to \cdots \to P_{1}\stackrel{f_{1}}\to P_0\stackrel{f_{0}}\to M\to 0$$ with $C$ a C-E Gorenstein projective module and each $P_j$ a C-E projective module. Hence we have a complete C-E projective resolution $T\stackrel{u}\to P\to M$. This gives a $\CE$-exact commutative diagram $$\xymatrix{
        0\ar[r] & C \ar[r] \ar@{=}[d] & P^0 \ar[r] \ar[d] & P^{-1} \ar[r] \ar[d] & \cdots \ar[r] & P^{-n+2} \ar[r] \ar[d] & P^{-n+1} \ar[r] \ar[d] & D \ar[d] \ar[r] & 0\\
        0\ar[r] & C \ar[r] & P_{n-1} \ar[r] & P_{n-2} \ar[r] & \cdots \ar[r] & P_{1} \ar[r] & P_{0} \ar[r] & M \ar[r] & 0\\
    }$$
with $D$ being C-E Gorenstein projective. It follows that there exists a $\CE$-exact sequence $G_M: 0\to P^0\to P_{n-1}\oplus P^{-1}\to \cdots \to P_{1}\oplus P^{-n+1}\to P_0\oplus D \stackrel{\beta}\to M \to 0$ with $\beta$ being a $\mathcal{G}$-precover of $M$. Now we construct a chain map $e:P_M\to G_M$  $$\xymatrix@C=1.3em{
    P_M:\ar[d]^{e} & \cdots\ar[r] & P^{2}\ar[d]^{0} \ar[r] & P^{1} \ar[r] \ar[d]^{d^{1}}& P_{n-1} \ar[r]\ar[d]^{(\smqty{1\\0})} & \cdots\ar[r] & P_0  \ar[r]\ar[d]^{(\smqty{1\\0})} & M \ar[r]\ar@{=}[d] & 0 \ar[r] & \cdots\\
    G_M: & \cdots\ar[r] & 0 \ar[r] & P^{0} \ar[r] & P_{n-1}\oplus P^{-1} \ar[r] & \cdots\ar[r] & P_0\oplus D \ar[r] & M \ar[r] & 0 \ar[r] & \cdots.
}$$ Then $\widehat{\Tor}_n^{\text{C-E},\mathcal{G},\text{C-E}(\mathcal{P})}(M,N)=H_{n+1}(M(e)\otimes N)$ and $\widehat{\Tor}_n^{\text{C-E}}(M,N)=H_n(T\otimes N)\cong H_{n+1}(T[1]\otimes N)$. We can check that there is a homotopy equivalence between $T[1]$ and $M(e)$. So $\widehat{\Tor}_n^{\text{C-E},\mathcal{G},\text{C-E}(\mathcal{P})}(M,N)\cong \widehat{\Tor}_n^{\text{C-E}}(M,N)$ for any $n\geq 1$.
\end{proof}

\begin{theorem}\label{thm-3}
    Let $M$ be a complex with $\text{\rm C-E Gpd}M<\infty$. For each complex $N$, there is an exact sequence $$\cdots\to \Tor^{\CE,\mathcal{G}}_2(M,N)\to \widehat{\Tor}_1^{\CE}(M,N) \to \Tor_1^{\CE}(M,N)\to \Tor^{\CE,\mathcal{G}}_1(M,N)\to 0.$$
\end{theorem}
\begin{proof}
    Let $P$ be a $\CE$-exact $\text{C-E}(\mathcal{P})$-resolution and $G$ be a $\CE$-exact $\mathcal{G}$-resolution of $M$. Let $u:P\to G$ be the chain map induced by $\Id_M$. Since both $P$ and $G$ are $\CE$-exact, the mapping cone $M(u)$ is also $\CE$-exact. It has the exact subcomplex $0\to M\stackrel{\Id_M}\longrightarrow M\to 0$. Forming the quotient one obtains the exact complex $M(u)$ that is the mapping cone of $u:P_{\bullet}\to G_{\bullet}$. The sequence $0\to G_{\bullet}\to M(u)\to P_{\bullet}[1]\to 0$ is split exact in each degree, so for each $N$ we have an exact sequence $0\to G_{\bullet}\otimes N\to M(u)\otimes N\to P_{\bullet}[1]\otimes N\to 0$. We obtain an exact sequence $$\cdots \to H_{n+1}(G_{\bullet}\otimes N)\to H_{n+1}(M(u)\otimes N)\to H_{n+1}(P_{\bullet}[1]\otimes N)\to H_n(G_{\bullet}\otimes N)\to\cdots.
    $$ Since $M(u)$ is exact and $-\otimes N$ is right exact, we follow that $H_1(M(u)\otimes N)=H_0(M(u)\otimes N)=0$. Then this exact sequence becomes
    \begin{align*}
            & \cdots \to H_{2}(G_{\bullet}\otimes N)\to H_{2}(M(u)\otimes N)\to H_{2}(G_{\bullet}[1]\otimes N)\to H_1(G_{\bullet}\otimes N)\to\\
            & 0 \to H_1(G_{\bullet}[1]\otimes N)\to H_0(G_{\bullet}\otimes N)\to 0.
    \end{align*}
    Combining with Proposition \ref{prop-8}, we have the exact sequence $$\cdots\to \Tor^{\text{C-E},\mathcal{G}}_2(M,N)\to \widehat{\Tor}_1^{\text{C-E}}(M,N) \to \Tor_1^{\text{C-E}}(M,N)\to \Tor^{\text{C-E},\mathcal{G}}_1(M,N)\to 0.$$
\end{proof}

\begin{definition}
    Let $N$ be a complex of finite C-E Gorenstein projective dimension. Consider a complete C-E projective resolution $V\to P\to N$ of $N$. We define $\overline{\Tor}_n^{\text{C-E}}(M,N)=H_n(M\otimes V)$ for any complex of right $R$-modules $M$ and $n\in\mathbb{Z}$.
\end{definition}

\begin{remark}\label{rmk-1}
    {\rm (1)} For any $n\geq 1$, the homology group $\overline{\Tor}_n^{\text{C-E}}(M,N)$ can also be computed by using $\CE$-exact C-E projective and C-E Gorenstein projective resolutions of the left $R$-module $N$. That is, if $P$ is a C-E projective resolution of $N$, and $G$ is a C-E Gorenstein projective resolution of $N$, and the chain map $v:P_{\bullet}\to G_{\bullet}$ is induced by $\Id_N$, then $\overline{\Tor}_n^{\text{C-E}}(M,N)=H_{n+1}(M\otimes M(v))$ for any $n\geq 1$.

    {\rm (2)} Similarly, there is an exact sequence $$\cdots\to \Tor^{\text{C-E},\mathcal{G}}_2(M,N)\to \overline{\Tor}_1^{\text{C-E}}(M,N) \to \Tor_1^{\text{C-E}}(M,N)\to \Tor^{\text{C-E},\mathcal{G}}_1(M,N)\to 0.$$
\end{remark}
It is still an open problem whether every Gorenstein projective module is Gorenstein flat, so we do not know whether the method in Theorem \ref{thm-4} can be applied to Tate homology. However, if $R$ is a left and right virtually Gorenstein ring, one can use Gillespie's
methods mentioned in Remark \ref{rmk-2}. More precisely, let $\mathfrak{M}$ and
$\mathfrak{M}'$ be the exact model structures
$$\mathfrak{M}=(\text{C-E}({\rm GProj}_R),\text{C-E}({\rm GProj}_R^{\perp}),
\Ch(R^{\circ}))$$ and
$$\mathfrak{M'}=(\text{C-E}(_R{\rm GProj}),\text{C-E}(_R{\rm GProj}^{\perp}),\Ch(R))$$ in
$\Ch(R^{\circ})$ and $\Ch(R)$,
respectively (where $\Ch(R^{\circ})$ denotes the category of chain complexes of right $R$-modules). If $G$ and $G'$ are C-E Gorenstein projective complexes in the two categories, respectively, then it follows that the homology groups $\ell\Tor_n^{\mathfrak{M}}(G,G')$ and $r\Tor_n^{\mathfrak{M'}}(G,G')$, in the sense of \cite[Definition 8.2]{G21}, coincide with $\widehat{\Tor}_n^{\text{C-E}} (G,G')$ and $\overline{\Tor}_n^{\text{C-E}}(G,G')$, respectively. By \cite[Proposition 8.6]{G21}, this yields the balance result for Tate homology of $\Tor$.

\begin{theorem}\label{thm-5}
    If $R$ is a left and right virtually Gorenstein ring, then $\widehat{\Tor}_n^{\CE}(M,N)\cong \overline{\Tor}_n^{\CE}(M,N)$ for any $n\in\mathbb{Z}$.
\end{theorem}
\begin{proof}
    The proof is similar to that of Theorem \ref{thm-4} and uses \cite[Proposition 8.6]{G21}.
\end{proof}

\end{document}